\documentclass[11pt,a4paper,english]{article}

\newcommand{\LL}{\mathbb{L}_X}
\newcommand{\R}{\mathbb{R}}

\newcommand{\beq}{\begin{equation}}
\newcommand{\eeq}{\end{equation}}
\newcommand{\bea}{\begin{eqnarray}}
\newcommand{\eea}{\end{eqnarray}}
\newcommand{\beas}{\begin{eqnarray*}}
\newcommand{\eeas}{\end{eqnarray*}}

\usepackage[latin1]{inputenc}
\usepackage{babel}
\usepackage[centertags]{amsmath}
\usepackage{amsfonts}
\usepackage{amssymb}
\usepackage{color}
\usepackage{amsthm}
\usepackage{newlfont}
\usepackage{graphicx}
\usepackage{dsfont}
\usepackage{geometry}
\usepackage{mathrsfs}
\usepackage{authblk}
\usepackage{verbatim}
\usepackage{todonotes}
\usepackage{hyperref}

\newtheorem{theorem}{Theorem}[section]
\newtheorem{lemma}[theorem]{Lemma}
\newtheorem{coroll}[theorem]{Corollary}
\newtheorem{prop}[theorem]{Proposition}
\newtheorem{definition}[theorem]{Definition}
\newtheorem{remark}[theorem]{Remark}
\newtheorem{ass}[theorem]{Assumption}

\def\RR{\mathbb R}
\def\FF{\mathbb F}
\def\EE{\mathsf E}
\def\PP{\mathsf P}

\def\eps{\varepsilon}
\def\cF{{\cal F}}
\def\cA{{\cal A}}

\def\cI{{\cal I}}
\def\cL{{\cal L}}
\def\cJ{{\cal J}}
\def\cO{{\cal O}}
\def\cT{{\cal T}}
\def\LL{{\mathbb{L} }}

\makeatletter  
\@addtoreset{equation}{section}
\def\theequation{\arabic{section}.\arabic{equation}}
\makeatother  
\begin{document}

\title{\textbf{Nash equilibria of threshold type \\ for two-player nonzero-sum games of stopping} \footnote{The first and the third authors were supported by EPSRC grants EP/K00557X/1 and EP/K00557X/2; financial support by the German Research Foundation (DFG) through the Collaborative Research Centre `Taming uncertainty and profiting from randomness and low regularity in analysis, stochastics and their applications' is gratefully acknowledged.}}

\author{Tiziano De Angelis\thanks{School of Mathematics, University of Leeds, Woodhouse Lane, Leeds LS2 9JT, United Kingdom; \texttt{t.deangelis@leeds.ac.uk}}\:\:\:\:Giorgio Ferrari\thanks{Center for Mathematical Economics, Bielefeld University, Universit\"atsstrasse 25, D-33615 Bielefeld, Germany; \texttt{giorgio.ferrari@uni-bielefeld.de}}\:\:\:\:John Moriarty\thanks{School of Mathematical Sciences, Queen Mary University of London, Mile End Road, London E1 4NS, United Kingdom; \texttt{j.moriarty@qmul.ac.uk}}}

\date{\today}
\maketitle

\textbf{Abstract.} This paper analyses two-player nonzero-sum games of optimal stopping on a class of linear regular diffusions with \emph{not non-singular} boundary behaviour (in the sense of It\^o and McKean (1974), p.\ 108). We provide sufficient conditions under which Nash equilibria are realised by each player stopping the diffusion at one of the two boundary points of an interval. The boundaries of this interval solve a system of algebraic equations. We also provide conditions sufficient for the uniqueness of the equilibrium in this class.
\medskip


\smallskip

{\textbf{MSC2010 subject classification}}: 91A05, 91A15, 60G40, 60J60, 35R35.


\section{Introduction}

Given a one-dimensional regular diffusion $X=(X_t)_{t \geq 0}$ on an interval $\mathcal{I}\subseteq \mathbb{R}$, we consider a two-player {\em Dynkin game} \cite{Dynkin01} in which player $i \in \{1,2\}$ chooses a stopping time $\tau_i$ in order to minimise the {\em payoff} $\mathcal{J}_i(\tau_1,\tau_2;x)$, where
\begin{align}
\label{functional0}
&\mathcal{J}_1(\tau_1,\tau_2;x):=\EE_x\Big[e^{-r\tau_1}G_1(X_{\tau_1})\mathds{1}_{\{\tau_1 < \tau_2\}} + e^{-r\tau_2}L_1(X_{\tau_2})\mathds{1}_{\{\tau_2 \leq \tau_1\}}\Big],\\[+3pt] 
\label{functional1}
&\mathcal{J}_2(\tau_1,\tau_2;x):=\EE_x\Big[e^{-r\tau_2}G_2(X_{\tau_2})\mathds{1}_{\{\tau_2 \le \tau_1\}} + e^{-r\tau_1}L_2(X_{\tau_1})\mathds{1}_{\{\tau_1 < \tau_2\}}\Big],
\end{align}
taking into account the stopping time chosen by player $j:=3-i$ (here $\EE_x$ denotes the expected value under the measure $\PP_x(\,\cdot\,)=\PP(\,\cdot\,|X_0=x)$). In particular we aim to provide sufficient conditions for the existence and uniqueness of Nash equilibria of the following {\em threshold type}:
\begin{align}\label{st-times}
\tau_1^*=&\tau_1(x_1^*),
\quad \tau_2^*=\tau_2(x_2^*) \quad \text{ for some } \quad 
x^*_1 < x^*_2,
\end{align}
where, for $x,z\in\cI$, we define $\PP_x$-a.s.~the stopping times
\begin{align}\label{st-times0}
\tau_1(z):=&\inf\{t> 0\,:\,X_t\le z \}\quad\text{and}\quad\tau_2(z):=\inf\{t> 0\,:\,X_t\ge z \}, 
\end{align} 
For $i=1,2$ we refer to $\mathcal{J}_i(\tau_1^*,\tau_2^*;x)$ as an \textit{equilibrium payoff}.
Our interest in such equilibria is guided by the seminal paper \cite{BensoussanFriedman}, where equilibrium stopping times are the hitting times of sets defined by free boundaries. 

In our game the stopping cost for player $i$ is equal to either $G_i(X_{\tau_i})$ or $L_i(X_{\tau_j})$ (continuously discounted at the rate $r>0$), depending on who stops first. While the existence and uniqueness of Nash equilibria are questions of fundamental interest, it is also valuable to investigate the {\em structure} of equilibria under general dynamics for $X$, a topic which has so far received relatively little rigorous mathematical treatment. 
In particular, having established the existence of an equilibrium structure such as \eqref{st-times} we may derive algebraic characterisations of the thresholds (that is, $x^*_1$ and $x^*_2$ in \eqref{st-times}) enabling straightforward numerical evaluation of both the equilibrium payoffs and stopping times. 

\subsection{Background and contribution}

In this section we briefly review the mathematical literature on Dynkin games (some of the economic literature will also be recalled in Section \ref{sec:conditions}) and place our contribution in this context. 

Zero-sum Dynkin games, in which $G_1=-L_2$ and $G_2=-L_1$, have a considerable  literature both in discrete and continuous time (see for instance \cite{Alarioetal}, \cite{Alvarez2008}, \cite{Bismut}, \cite{CvitanicKar}, \cite{DeAFe14}, \cite{EkstromVilleneuve}, \cite{Kifer}, \cite{Kypri}, \cite{LarakiSolan}, \cite{PeskirGame}, \cite{SirbuShreve}, \cite{TouziVieille} and references therein) and in several papers the structure of equilibria has been studied, beyond the question of their existence and uniqueness. In contrast
the literature on nonzero-sum Dynkin games focuses largely on the existence of Nash equilibria. 

Such existence results have been given in both the discrete and the continuous time settings, using a variety of sufficient conditions and methodologies. In discrete time the approaches applied include Tarski's fixed point theorem \cite{Mamer} and martingale methods combined with a fixed point theorem for monotone mappings \cite{Morimoto}; see also \cite{Ohtsubo2} and \cite{ShmayaSolan}. In continuous time the methodologies employed include quasi-variational inequalities \cite{Nagai}, the general theory of stochastic processes \cite{Etorneau} and the potential theory of Ray-Markov processes \cite{CattiauxLep}; see also \cite{HamZha10}, \cite{HamadeneHassani},  \cite{LarakiSolan} and \cite{LarakiSolan2}. 
Beyond the question of existence, equilibrium payoffs (although not strategies) are constructed by backward induction in \cite{Ohtsubo}. Additional technical questions arising in the continuous time setting, concerning appropriate notions of subgame-perfect equilibrium, are considered in \cite{RiSte15}.

In the present paper we take $X$ to be a weak solution of a stochastic differential equation (SDE) on an interval $\mathcal{I} =(\underline x,\overline x)$ and employ probabilistic methods so that, unlike in analytical settings such as that of \cite{BensoussanFriedman}, the coefficients of this SDE are only assumed to be continuous. We provide sufficient conditions on the functions $L_i$ and $G_i$ for the existence of a Nash equilibrium whose strategies have the structure \eqref{st-times}, together with sufficient conditions for uniqueness of the equilibrium in this class. Our method is the following: we show that if player 1 stops at the hitting time of a half-line $(\underline{x},z_1]$ then player 2's best reply is the hitting time to a half-line $[z_{2,1},\overline{x})$ where $z_{2,1}>z_1$ depends on $z_1$. The same arguments produce symmetric results if we start by letting player 2 stop at the hitting time of a half-line $[z_{2},\overline{x})$. The key point is then to show that there exists a fixed point, i.e.~a couple of intervals $(\underline{x},x_1^*]$ and $[x_2^*,\overline{x})$ such that the corresponding hitting times are the best replies to one another and \eqref{st-times} forms an equilibrium. To construct each player's best reply to the opponent's stopping rule we draw on the geometric characterisation of $r$-excessive mappings due to Dynkin \cite{Dynkin} and later generalised in \cite{DayKar}.

Our work complements recent related work by Attard \cite{Attard2} (see also \cite{Attard1}), 
which became available during the final drafting of the present paper. In \cite{Attard2} the structure of Nash equilibria is studied for regular linear diffusions absorbed at either $0$ or $1$. Here instead we consider regular linear diffusions on an interval $\cI\subseteq\RR$ killed at a (possibly state-dependent) rate and cover all boundary behaviours which are not non-singular. 
The methodology in \cite{Attard2} differs from our approach, which is to construct the solution by taking into account the geometry of the stopping cost functions. In contrast in \cite{Attard2} the equilibrium payoffs are hypothesised to satisfy the so-called \emph{double smooth-fit} principle, according to which they are continuously differentiable at the corresponding equilibrium stopping threshold $x^*_i$. Based on this principle two coupled free boundary problems are formulated and solved explicitly, producing functions $u$ and $v$ which can be verified as the equilibrium payoffs of the game. In the present paper we observe  double smooth fit in some equilibria (in particular see Section \ref{Nash-natural}), but we also find equilibria outside this setting (see Sections \ref{entrance}, \ref{exit} and \ref{relax}). Finally a more technical difference is that due to the absorbing behaviour at 0 and 1, the main result in \cite{Attard2} (Theorem 4.1) assumes that $G_i(0) = L_i(0)$ and $G_i(1) = L_i(1)$, $i=1,2$  (here we use our notation for the payoffs). 
We allow instead the limiting behaviour given in \eqref{lims} and in \eqref{ass:lim} below for $G_i$ and $L_i$, respectively.

\subsection{Outline of main results}\label{intro:mainres}
In order to present the main results we first note properties of the underlying regular diffusion $X$. In general the behaviour of the process $X$ at the boundaries of $\cI\subseteq\RR$ may be of several types \cite{ItoMcKean} and we will assume that the upper endpoint of $\mathcal{I}$ is \emph{natural}, while the lower one is \emph{not non-singular}: that is, either \emph{natural, exit-not-entrance} or \emph{entrance-not-exit} (see for instance \cite{BS}, Ch.~2, pp.~18--20). For the unfamiliar reader the terminology is explained in more detail in Section \ref{sec:diffusion} where other analytical properties of $X$ are also addressed.

Beginning with the case of natural boundaries (which includes Brownian motion, geometric Brownian motion and the Ornstein-Uhlenbeck process), we establish sufficient conditions on $G_i$ and $L_i$ (see the next section for details)
for the existence of a Nash equilibrium $(\tau^*_1,\tau^*_2)$ of the threshold type \eqref{st-times}. Under these conditions the {\em smooth fit} principle holds for the equilibrium payoff $x \mapsto \mathcal{J}_i(\tau^*_1,\tau^*_2;x)$ at the corresponding equilibrium threshold $x^*_i$ (i.e.~the payoff is continuously differentiable at $x^*_i$). These thresholds may therefore be characterised by a system of algebraic equations. We then show that if the functions $L_i$ are assumed to have appropriate smoothness, we may also provide sufficient conditions for the uniqueness of the Nash equilibrium amongst strategies of the threshold type \eqref{st-times}. 

Analogous results are obtained in the case when the lower endpoint is either an exit-not-entrance or an entrance-not-exit boundary, thus addressing Bessel processes (depending on their parameters) and related processes such as the CIR (Cox-Ingersoll-Ross) and CEV (constant elasticity of variance) process. In these settings we also find equilibria in which one of the two players never stops, and equilibria with a structure possibly more complex than \eqref{st-times} depending on the initial value of $X$ (see Proposition \ref{prop:S}).

We also indicate in Appendix \ref{sec:rdisc} the extension to a state dependent discount factor. Other combinations of natural, exit-not-entrance and entrance-not-exit boundaries may be addressed via the methods of this paper (indeed this is immediate by symmetry when the lower boundary is natural and the upper one not non-singular). 

\subsubsection{Conditions on the problem data}\label{sec:conditions}

We consider stopping costs $L_i$ and $G_i$ fulfilling suitable assumptions of integrability and smoothness (cf.\ Definitions \ref{def:sets0}, \ref{def:sets}). Moreover, they satisfy the sufficient conditions applied below, which are motivated by the threshold type equilibrium structure \eqref{st-times}. For $i=1,2$ these are:
\begin{enumerate}
\item [(a)] $L_i < G_i$, 
\item[(b)] $\overline{\Gamma}_1 \cap \overline{\Gamma}_2 = \emptyset$, where $\overline{\Gamma}_i$ denotes the closure of $\Gamma_i:=\{x: (\mathbb{L}_X - r)G_i(x) > 0\}$ and $\LL_X$ denotes the infinitesimal generator of $X$,
\item[(c)] The equation $(\mathbb{L}_X - r)G_i(x) = 0$ has a single root.
\end{enumerate}

In our setup player $i$ chooses a stopping time $\tau_i$, or equivalently a {\em pure strategy} (see, for example, \cite{RiSte15} for discussion on pure and mixed strategies in continuous time). Our requirement (a) specifies that each player locally has an incentive to act second: in the context of stopping games, this is a \emph{war of attrition} (see, e.g., \cite{Murto}). It is worth mentioning here that in the opposite situation, in which $L_i > G_i$, each player locally has an incentive to act first and the game belongs to the class of \emph{preemption games}, see for example \cite{FT85} for a deterministic setting and \cite{Anderson,Steg2016} for a stochastic framework. In the literature on preemption games, equilibria are usually realised in mixed rather than pure strategies. Requirement (a) is therefore reasonable in a study of threshold-type strategies, which are pure strategies constructed from hitting times.

Condition (b) addresses the cost functions $G_i$, which are the costs for stopping first. An argument using Dynkin's formula, which is standard in optimal stopping and is also provided in Appendix \ref{2usefulOS}, establishes that player $i$ will not stop on the set $\{x: (\mathbb{L}_X - r)G_i(x) < 0\}$. Our requirement (b) therefore means that locally at least one player is incentivised to wait rather than accepting the stopping cost. This implies that the players do not stop simultaneously. 

Motivated by the latter argument, let us temporarily fix the strategy of player $i$. Then it is reasonable to suppose that if the function $x \mapsto g_j(x):=(\mathbb{L}_X - r)G_j(x)$, $j=3-i$, changes sign many times, the optimal strategy for player $j$ (given player $i$'s strategy) may involve several disjoint regions on which it is optimal to stop. Since condition (c) ensures that the function $g_j$ changes sign at most once on $\cI$, this suggests that the optimal strategy for player $j$ (given player $i$'s strategy) should be to stop the process upon hitting a connected set.
Indeed this type of condition is commonly used in the literature on optimal stopping problems in order to ensure that the solution is a stopping time of threshold type.
\vspace{+6pt}

In principle our techniques may also apply under conditions other than (a)--(c), in which case equilibria with other structures can arise. This point is illustrated in Section \ref{relax}, where condition (b) is replaced by the following alternative:
\begin{enumerate}
\item[(d)] Either $\Gamma_1\subset \Gamma_2$ or $\Gamma_2\subset \Gamma_1$.
\end{enumerate} 
In this case it is necessary to slightly generalise the structure of \eqref{st-times} as one of the players may never stop. Such equilibria have been obtained in the economics literature for instance by Murto \cite{Murto} but in Section \ref{relax} we consider more general specifications of the stochastic process $X$ and of the stopping costs $L_i$ and $G_i$ than those used in \cite{Murto}.

Note that the threshold type structure we study has some degree of overlap with that of some zero-sum games of optimal stopping (see, e.g., \cite{Alvarez2008}). However a characterisation of the equilibrium stopping thresholds in our game cannot be achieved via methods usually employed for zero-sum games. Indeed in this paper we deal with the joint optimisation of the coupled system of payoffs \eqref{functional0}-\eqref{functional1}. The latter reduces to a simpler problem with a single payoff in the zero-sum case. From a PDE point of view this can be understood by noticing that a zero-sum game is usually associated to a variational inequality (see, e.g.~\cite{Alvarez2008}) whereas a nonzero-sum game must be associated to a system of coupled variational inequalities (see, e.g.~\cite{BensoussanFriedman}).


\vspace{2mm}

The rest of the paper is organised as follows. Section \ref{sec:setting} introduces the nonzero-sum Dynkin game, together with the dynamics and our assumptions and sufficient conditions on the stopping costs. Existence and uniqueness results for Nash equilibria of threshold type are proved in Section \ref{sec:construction} for different combinations of boundary behaviour. In Section \ref{relax} we consider slightly weaker assumptions on the stopping costs. In the Appendix we generalise our results to state dependent discount factors and provide some auxiliary proofs.


\section{Setting}
\label{sec:setting}

We begin by formally stating the game presented in the Introduction and by providing a rigorous definition of a Nash equilibrium in Section \ref{sec:game}. Then we describe the class of diffusions involved in the optimisation (see Section \ref{sec:diffusion}), whose analytical properties are finally used in Sections \ref{sec:H} and \ref{sec:payoffs} to characterise the class of stopping costs to be used in the game, i.e.~$G_i$, $L_i$, $i=1,2$ in \eqref{functional0} and \eqref{functional1}.

\subsection{The nonzero-sum Dynkin game}
\label{sec:game}
On a complete filtered probability space $(\Omega, \mathcal{F},\PP)$ equipped with a filtration $\FF=(\mathcal{F}_t)_{ t \geq 0}$ which satisfies the standard assumptions, we consider a real-valued diffusion process $X:=(X_t)_{t\ge 0}$. Its state space is an interval $\cI\subseteq \RR$ and we denote the \emph{explosion time} by 
\begin{align}\label{sigmaI}
\sigma_{\,\cI}:=\inf\{t\ge0\,:\,X_t\notin\cI\}.
\end{align}
Further details concerning the diffusion $X$ will be provided in the next section.

In order to formally introduce the \emph{two-player nonzero-sum game of stopping} we denote
\begin{align*}
\mathcal{T}:=\{\tau:\,\text{$\tau$ is an $\FF$-stopping time and $\tau\le\sigma_\cI$ a.s.}\},
\end{align*}
and when no confusion may arise we also denote player 1 by $P_1$ and player 2 by $P_2$. 
For a given $\tau_2\in\cT$, player $P_1$ aims at \emph{minimising} the payoff $\mathcal{J}_1(\tau_1,\tau_2;x)$ in \eqref{functional0} by optimally choosing their stopping time $\tau_1\in\cT$. Analogously, given $\tau_1\in\cT$, $P_2$ chooses the stopping time $\tau_2\in\cT$ in order to minimise $\mathcal{J}_2(\tau_1,\tau_2;x)$ in \eqref{functional1}. 

In order to cover the events $\omega$ for which $\sigma_\cI(\omega)=+\infty$, for any real-valued Borel function $f$ and any $\tau\in\cT$, we set  
\begin{align}\label{limit}
e^{-r\tau}f(X_\tau)\mathds{1}_{\{\tau=+\infty\}}=0\quad\text{$\PP_x$-a.s.~for all $x \in \cI$}.
\end{align}
Before proceeding further we provide the definition of Nash equilibrium.
\begin{definition}
\label{def:Nash}
For $x \in \cI$ we say that a couple $(\tau_1,\tau_2) \in \mathcal{T} \times \mathcal{T}$ is a Nash equilibrium for the two-player nonzero-sum game of optimal stopping, started at $x$, if and only if
\begin{equation}
\label{Nashequilibrium}
\left\{
\begin{array}{ll}
\mathcal{J}_1(\tau_1,\tau_2;x) \leq \mathcal{J}_1(\rho,\tau_2;x), \quad \forall\,\rho \in \mathcal{T}, \\[+5pt]
\mathcal{J}_2(\tau_1,\tau_2;x) \leq \mathcal{J}_2(\tau_1,\rho;x), \quad \forall\,\rho \in \mathcal{T}.
\end{array}
\right.
\end{equation}
We also say that $v_i(x):=\mathcal{J}_i(\tau_1,\tau_2;x)$ is the corresponding equilibrium payoff for the $i$-th player. Further, if the couple $(\tau_1,\tau_2)$ is an equilibrium in the game started at $x$ for each $x\in\cI$, we simply say that $(\tau_1,\tau_2)$ is a Nash equilibrium. 
\end{definition}

\subsection{The underlying diffusion}
\label{sec:diffusion}

Let $B=(B_t)_{t\geq 0}$ be a one dimensional standard Brownian motion on $(\Omega, \cF,\PP)$ adapted to $\FF$, then our diffusion $X$ is defined as follows. The triple $(\Omega, \cF,\PP),\FF,(X,B)$ is a weak solution of the stochastic differential equation (SDE)
\begin{equation}
\label{state:X}
dX_t=\mu(X_t)dt+\sigma(X_t)dB_t,\qquad X_0=x\in\cI,
\end{equation}
for some Borel-measurable functions $\mu,\sigma: \mathbb{R} \to \mathbb{R}$ to be specified.
To account for the dependence of $X$ on its initial position, from now on we shall write $X^x$ where appropriate and $\PP_x$ to refer to the probability measure such that $\PP_x(\cdot) = \PP(\cdot| X_0=x)$, $x \in \mathcal{I}$. Throughout the paper we will equivalently use the notations $\EE[f(X^x_t)]$ and $\EE_x[f(X_t)]$, $f:\RR \to \RR$ Borel-measurable and integrable, to refer to expectations under the measure $\PP_x$.

We denote by $\overline{\cI}$ the closure of $\mathcal{I}$ and assume that $\mathcal{I}=(\underline{x},\overline{x}) \subseteq \R$ so that $\underline{x}$ and $\overline{x}$ are (not necessarily finite) boundary points for $X$. The upper boundary point $\overline{x}$ is assumed to be natural, whereas the lower one $\underline{x}$ is either natural, exit-not-entrance or entrance-not-exit (see for instance Ch.~2, pp.~18--20, of \cite{BS} for a characterisation of the boundary behaviour of diffusions). We recall that $\overline{x}$ (or equivalently $\underline{x}$) is natural if the process cannot start from there and when starting from $x\in\cI$ it cannot reach $\overline{x}$ (resp.~$\underline{x}$) in finite time; $\underline{x}$ is exit-not-entrance if the process cannot start from $\underline{x}$ but can reach it in finite time (hence $\sigma_\cI<+\infty$ with positive probability); finally $\underline{x}$ is entrance-not-exit if the process can start from $\underline{x}$ but it cannot reach it in finite time when started from $x\in\cI$.

For the coefficients of the SDE \eqref{state:X} we make the following assumption, which will hold throughout the paper.
\begin{ass}\label{ass:D2}
The functions $\mu$ and $\sigma$ are continuous in $\cI$ with $\sigma^2>0$ in $\mathcal{I}$.
\end{ass}
As a consequence of the above assumption one has that for every $y \in \mathcal{I}$ there exists $\varepsilon_o>0$ such that
\begin{equation*}
\label{LI}
\int_{y-\varepsilon_o}^{y+\varepsilon_o}\frac{1 + |\mu(\xi)|}{|\sigma(\xi)|^2}\,d\xi < +\infty.
\end{equation*}
The latter guarantees that \eqref{state:X} has indeed a weak solution that is unique in the sense of probability law (up to the time $\sigma_\cI$, cf.\ \cite{KS}, Ch.\,5.5).

We now recall some basic analytical properties of diffusions, which are also going to be used later on to characterise the functions $G_i$, $L_i$ appearing as stopping costs in the game (recall \eqref{functional0} and \eqref{functional1}). We refer the reader to Ch.\ 2 of \cite{BS} for a detailed exposition.
Under Assumption \ref{ass:D2}, the diffusion process $X$ is regular in $\mathcal{I}$; that is, if 
\begin{align}\label{tau-y}
\tau(y):=\inf\{t\geq0: X_t=y\}
\end{align}
one has $\PP_x(\tau(y)<\infty) > 0$ for every $x$ and $y$ in $\mathcal{I}$ so that the state space cannot be decomposed into smaller sets from which $X$ cannot exit. The continuity of $\mu$ and $\sigma$ imply that the scale function has derivative 
$$S'(x):=\exp{\bigg(-\int_{x_o}^x \frac{2\mu(\xi)}{\sigma^2(\xi)}d\xi\bigg)}, \quad x\in \cI,$$
for any fixed reference point $x_o \in \cI$, and the speed measure has density
$$m'(x):= \frac{2}{\sigma^2(x)S'(x)}, \quad x \in \cI.$$
We define the infinitesimal generator $\LL_X$ of $X$ by
$$\big(\LL_X u\big)(x):=\frac{1}{2}\sigma^2(x)u''(x) + \mu(x)u'(x), \quad \,\,x\in\mathcal{I},$$
for any $u \in C^2(\mathcal{I})$. Then, for fixed $r>0$, under Assumption \ref{ass:D2} there always exist two linearly independent, strictly positive solutions of the ordinary differential equation $\LL_X u = r u$ satisfying a set of boundary conditions based on the boundary behaviour of $X$ (see, e.g., pp.\ 18--19 of \cite{BS}). These functions span the set of solutions of $\LL_X u = r u$ and are uniquely defined up to multiplication if one of them is required to be strictly increasing and the other one to be strictly decreasing. We denote the strictly increasing solution $\psi_r$ and the strictly decreasing one $\phi_r$.
For $x,y\in\cI$ and $\tau(y)$ as in \eqref{tau-y} one has
\begin{align}\label{eqn:lapl}
\EE_x\big[e^{-r\tau(y)}\big]=\left\{
\begin{array}{ll}
\frac{\psi_r(x)}{\psi_r(y)}, & x<y,\\[+5pt]
\frac{\phi_r(x)}{\phi_r(y)}, & x>y.
\end{array}
\right.
\end{align}

We recall that the Wronskian
\begin{equation}
\label{Wronskian}
W:= \frac{\psi_r'(x)\phi_r(x) - \phi'_r(x)\psi_r(x)}{S'(x)}, \quad x \in \cI,
\end{equation}
is a positive constant and we introduce the Green function
$$r(x,y):= W^{-1} \cdot \left\{
\begin{array}{ll}
\psi_r(x)\phi_r(y), & x \leq y,\\[+4pt]
\phi_r(x)\psi_r(y), & x \geq y.
\end{array}
\right.
$$
The latter can be used to obtain the representation formula for the resolvent
\begin{equation}
\label{resolvent}
\EE_x\bigg[\int_0^{\sigma_\cI} e^{-rt} f(X_t) dt\bigg] = \int_{\cI} f(y)r(x,y)m'(y) dy, \quad x \in \cI,
\end{equation}
which holds for any continuous function $f:\cI\to\RR$ such that the integrals are well defined. Moreover the following useful equations hold for any $\underline{x}<a<b<\overline{x}$ (cf.~par.~10, Ch.~2 of \cite{BS}):
\begin{equation}
\label{psiphiproperties3}
\frac{\psi'_r(b)}{S'(b)} - \frac{\psi'_r(a)}{S'(a)}= r\int_{a}^{b}\psi_r(y)m'(y) dy, \qquad \frac{\phi'_r(b)}{S'(b)}-\frac{\phi'_r(a)}{S'(a)} = r\int_{a}^{b}\phi_r(y)m'(y) dy.
\end{equation}

\subsection{Classes of stopping cost functions}\label{sec:H}
In order to clarify the assumptions concerning the stopping costs $G_i$, $L_i$, $i=1,2$ appearing in \eqref{functional0} and \eqref{functional1} we need first to introduce the class of functions below.
\begin{definition}\label{def:sets0}
Let $\cA$ be the class of real valued functions $H \in C^2(\cI)$ such that
\begin{align}
\label{lims}&\lim_{x\to\underline{x}}\frac{H}{\phi_r}(x)=0,\:\:\lim_{x\to\overline{x}}\frac{H}{\psi_r}(x)=0\\[+4pt]
\label{lims2}&\quad\text{and}\quad\EE_x\bigg[\int_0^{\sigma_\cI} e^{-rt} \big|h(X_t)\big|dt\bigg] < \infty
\end{align}
for all $x\in\cI$ and with $h(x):= (\LL_X H - rH)(x)$. 
\end{definition}
\noindent In this paper, elements of $\cA$ will be often denoted by $H$ and then the corresponding lower case letter $h$ will denote the function $h(x):=(\LL_XH-rH)(x)$.

We provide some formulae for functions in $\cA$ which will be useful in the rest of the paper. Using It\^o's formula, \eqref{lims} and standard localisation arguments one can show that for $H\in\cA$ we have
\begin{equation}
\label{Dynkin}
\displaystyle H(x)=-\EE_x\bigg[\int_0^{\sigma_{\cI}} e^{-rt} h(X_t) dt\bigg],\quad x\in\cI.
\end{equation}
Then applying the representation \eqref{resolvent} we get the equivalent expression
\begin{equation}
\label{resolvent2}
\displaystyle H(x)= -W^{-1}\bigg[\phi_r(x)\int_{\underline{x}}^x \psi_r(y)h(y) m'(y)dy + \psi_r(x)\int_x^{\overline{x}} \phi_r(y)h(y) m'(y)dy\bigg]
\end{equation}
and straightforward calculations also give
\begin{align}\label{res3}
\Big(\frac{H}{\phi_r}\Big)'(x)=-\frac{1}{W}\left(\frac{\psi_r}{\phi_r}\right)'(x)\int_x^{\overline{x}}\phi_r(y)h(y)m'(y)dy.
\end{align}

For our study we also consider the following subsets of $\cA$. 
\begin{definition}\label{def:sets}
We say that $H\in\cA$ lies in the class $\cA_1$ if $h(\,\cdot\,)$ has a unique zero at $x_h\in\cI$ and $\liminf_{x \to \underline x}h(x)>0$ and $\limsup_{x \to \overline x}h(x)<0$.
Alternatively we say that $H\in\cA_2$ if $-H\in\cA_1$.
\end{definition}

Several proofs below use a geometric approach to optimal stopping which requires the following change of variables. As in \cite{DayKar}, eq.\ (4.6), we define the strictly increasing function
\begin{equation}
\label{def-F}
F_{r}(x):=\frac{\psi_{r}(x)}{\phi_{r}(x)}, \qquad x \in \cI,
\end{equation}
together with its inverse function $F_{r}^{-1}$ and for any continuous real function $H$ on $\mathcal{I}$ we set
\begin{align}\label{def:hat}
\hat{H}(y):=\left\{
\begin{array}{ll}
\Big(\frac{H}{\phi_r}\Big)\circ F_r^{-1}(y), & y > 0,\\
0, & y=0.
\end{array}
\right.
\end{align}
In what follows (see, e.g., Lemma \ref{lemm:concavity}) for $H\in\cA_i$, $i=1,2$ we denote $$\hat{y}_h:=F_r(x_h)$$

For the benefit of the unfamiliar reader we provide in Appendix \ref{app:convex} a proof of the next well known result (see also Section $6$, p.\ 192 in \cite{DayKar}). 
\begin{lemma}\label{lem:conv}
Let $x_1,x_2\in\cI$ and set $y_i:= F_r(x_i)$, $i=1,2$. Moreover let $H\in C^2(\cI)$ and define $\hat{H}$ as in \eqref{def:hat} and $h:=(\LL_X-r)H$. Then  
\begin{align}
\text{$\hat{H}(y)$ is strictly convex on $(y_1,y_2)$}\iff\text{$h(x)>0$ on $(x_1,x_2)$}.
\end{align} 
\end{lemma}

\subsection{Sufficient conditions on the stopping costs and notation}\label{sec:payoffs}

Here we formulate the statements (a), (b) and (c) of Section \ref{sec:conditions} in the above setting. We will show in Section \ref{sec:construction} that these conditions are sufficient for the existence of Nash equilibria of threshold type \eqref{st-times}.
It is convenient to recall the notation $\overline{\Gamma}_i$ for the closure of the sets $\Gamma_i:=\{x\in\cI\,:\,(\LL_X-r)G_i(x)>0\}$, $i=1,2$. 
\begin{ass}
\label{ass:flessi}
\begin{itemize}
\item[(i)] For $i=1,2$ we have $L_i,\,G_i\in C(\cI;\RR)$ with $L_i<G_i$ on $\cI$;
\item[(ii)] $G_1\in\cA_1$ and $G_2\in\cA_2$ with $\overline{\Gamma}_1\cap\overline{\Gamma}_2=\emptyset$;
\item[(iii)] For $i=1,2$ we have
\begin{align}\label{ass:lim}
\limsup_{x\to\underline{x}}\Big|\frac{L_i}{\phi_r}\Big|(x)<+\infty\:\:\:\:\text{and}\:\:\:\:
\limsup_{x\to\overline{x}}\Big|\frac{L_i}{\psi_r}\Big|(x)<+\infty.
\end{align}
\end{itemize}
\end{ass}
In fact parts $(i)$ and $(ii)$ slightly refine conditions (a), (b) and (c) of Section \ref{sec:conditions}, since we now require the sign of $(\LL_X-r)G_i$ to be asymptotically non-zero at the endpoints of $\cI$. We have introduced condition $(iii)$ to ensure the finiteness of the game's payoffs (see e.g.~\cite{DayKar}).

It is useful to introduce also some notation related to the above assumptions on the stopping costs.
We recall \eqref{def:hat} and for $i=1,2$ we set $\hat{G}_i$ and $\hat{L}_i$ to be the transformations of $G_i$, $L_i$. 
\begin{definition}[\textbf{Notation}]
\label{def:inflection}
For $i=1,2$ and $G_i\in\cA_i$ we define
\begin{enumerate}
\item[1.] $g_i(x):=(\LL_X-r)G_i(x)$, $x\in\cI$;
\item[2.] $\hat{x}_i$ the unique point at which the sign of $g_i(x)$ changes and $\hat{y}_i:=F_r(\hat{x}_i)$;
\item[3.] $\overline{y}_i$ the unique stationary point of $\hat{G}_i$ in $(0,\hat{y}_i)$, whenever it exists.
\end{enumerate}
For $i=1,2$ and $L_i\in\cA_i$ we define
\begin{enumerate}
\item[4.] $\ell_i(x):=(\LL_X-r)L_i(x)$, $x\in\cI$;
\item[5.] $\check{x}_i$ the unique point at which the sign of $\ell_i(x)$ changes and $\check{y}_i:=F_r(\check{x}_i)$;
\item[6.] $\widetilde{y}_i$ the unique stationary point of $\hat{L}_i$ in $(0,\check{y}_i)$, whenever it exists.
\end{enumerate}
\end{definition}

\noindent Notice that $\hat{G}_i$ and $\hat{L}_i$ as in Definition \ref{def:inflection} have at most one stationary point in $(0,\hat{y}_i)$ and $(0,\check{y}_i)$, respectively, due to Lemma \ref{lem:conv}. Note also that in this setting 
$$\overline{\Gamma}_1\cap \overline{\Gamma}_2=\emptyset\iff\hat{x}_1 < \hat{x}_2.$$

\begin{remark}
\label{rem:A1A2}
For natural and entrance-not-exit boundaries we have $\phi_r(x)\uparrow +\infty$ as $x\downarrow \underline{x}$ and $\psi_r(x)\uparrow +\infty$ as $x\uparrow \overline{x}$ so that bounded functions $G_i$ satisfy \eqref{lims}, for example. In the case of an exit-not-entrance boundary which is explored in Section \ref{exit}, however, \eqref{lims} is more restrictive and so it is relaxed in the latter section, yielding an additional term in \eqref{Dynkin} (cf. \eqref{res2a}). We also note that all the results in this paper remain true if in the definition of $\cA$ the regularity of $H$ is weakened by requiring $H\in W^{2,\infty}_{loc}(\cI)$.
\end{remark}
 

\section{Construction of Nash equilibria}
\label{sec:construction}

In this section we develop our existence and uniqueness results under different combinations of diffusion boundary behaviour. We are then able to provide an algebraic characterisation of the optimal thresholds, as a system of two equations in two unknowns (or reducing in special cases to one equation in one unknown, with another threshold formally located at one of the endpoints $\underline x$ and $\overline x$). We begin in Section \ref{Nash-natural} under the assumption that the endpoints $\underline{x}<\overline{x}$ of $\cI$ are natural for $X$, then consider an  entrance-not-exit lower boundary $\underline{x}$ in Section \ref{entrance} and an exit-not-entrance lower boundary in Section \ref{exit}.

\subsection{The case of natural boundaries}
\label{Nash-natural}

When $\underline{x}$ and $\overline{x}$ are both natural boundary points we have (see par.~10, Sec.~2 of \cite{BS}):
\begin{equation}
\label{psiphiproperties1}
\lim_{x \downarrow \underline{x}}\psi_r(x) = 0,\,\,\,\,\lim_{x \downarrow \underline{x}}\phi_r(x) = \infty,\,\,\,\,\lim_{x \uparrow \overline{x}}\psi_r(x) = \infty,\,\,\,\,\lim_{x \uparrow \overline{x}}\phi_r(x) = 0,
\end{equation}
\begin{equation}
\label{psiphiproperties2}
\lim_{x \downarrow \underline{x}}\frac{\psi'_r(x)}{S'(x)} = 0,\,\,\,\,\lim_{x \downarrow \underline{x}}\frac{\phi'_r(x)}{S'(x)} = -\infty,\,\,\,\,\lim_{x \uparrow \overline{x}}\frac{\psi'_r(x)}{S'(x)} = \infty,\,\,\,\,\lim_{x \uparrow \overline{x}}\frac{\phi'_r(x)}{S'(x)} = 0.
\end{equation}

The following lemma provides geometric properties associated with the classes $\cA_1$, $\cA_2$ of Definition \ref{def:sets}, and is proved in the Appendix.
\begin{lemma}
\label{lemm:concavity}
Let $H\in \cA_1$ (respectively $\cA_2$). Then $\hat{H}$:
\begin{itemize}
	\item[  i)] is strictly convex (resp. concave) on $(0, \hat{y}_h)$ and strictly concave (resp. convex) on $(\hat{y}_h, \infty)$,
	\item[  ii)] satisfies $\hat{H}(0+)=0$ and $\hat{H}'(0+)=-\infty$ (resp. $+\infty$);
	\item[  iii)] has a unique global minimum (resp. maximum) in $(0, \hat{y}_h)$ and $\lim_{y\to\infty}\hat{H}(y)=+\infty$ (resp. $-\infty$); finally $\hat{H}$ is monotonic increasing (resp. decreasing) on $(\hat{y}_h,+\infty)$.
\end{itemize}
\end{lemma}

In order to prove our main results, for $i=1,2$ and $u,\,v > 0$ let us introduce the functions
\begin{align}
\label{eq:ell}
\cL_i(u,v):=\hat{G}_i(u) - \hat{L}_i(v) - \hat{G}_i'(u)(u - v).
\end{align}

\begin{theorem}[\textbf{Existence of an equilibrium}]
\label{thm:construction}
Under Assumption \ref{ass:flessi} there exists a solution $(y_1^*,y_2^*)$ of the problem:
\begin{equation}
\label{system-construction}
\text{Find $(y_1,y_2)\in (0, \hat{y}_1) \times (\hat{y}_2, +\infty)$ such that}\:\:
\left\{
\begin{array}{l}
\cL_1(y_1,y_2)=0,\\[+3pt]
\cL_2(y_2,y_1)=0.
\end{array}
\right.
\end{equation}
Writing $x^*_1:=F_r^{-1}(y_1^*)\in(\underline{x},\hat{x}_1)$ and $x^*_2:=F_r^{-1}(y_2^*)\in(\hat{x}_2,\overline{x})$ and recalling \eqref{st-times0}, the couple 
\begin{equation}
\label{Nashform}
\tau_1^*=\tau_1(x_1^*), \quad \tau_2^*=\tau_2(x_2^*)
\end{equation}
is a Nash equilibrium.
\end{theorem}
\begin{proof}
The proof is divided into three steps. In Step 1 we assume that $P_1$ picks a stopping time $\tau_1(z)$ as defined in \eqref{st-times0}, for some $z\in(\underline{x},\hat{x}_1)$. We then construct $P_2$'s best reply, 
showing that it has the form $\tau_2(x_2)$ as defined in \eqref{st-times0}, for some $x_2\in(\hat{x}_2,\overline{x})$ which depends on $z$.
Step 2 reverses the roles of the two players, and in Step 3 we combine these results to construct a Nash equilibrium. 
\vspace{+5pt}

\noindent\emph{Step 1 (Player $2$'s best reply)}. Given $P_1$'s choice $\tau_1(z)$ described above, $P_2$ is faced with an optimal stopping problem of the form
\begin{equation}
\label{OS-P_2}
\inf_{\tau \in \mathcal{T}}\EE_x\Big[e^{-r\tau}G_2(X_{\tau})\mathds{1}_{\{\tau \le \tau_1(z)\}} + L_2(X_{\tau_1(z)}) e^{-r\tau_1(z)}\mathds{1}_{\{\tau > \tau_1(z)\}}\Big].
\end{equation}
Setting $\zeta:=F_r(z)$, it is shown in Proposition \ref{prop:app-OS} that if the equation
\begin{align}\label{smf00}
\cL_2(\,\cdot\,,\zeta) = 0
\end{align}
has a solution $y_2(\zeta) \in (\hat{y}_2,+\infty)$ (which is therefore unique), then the stopping time $\tau_2(x_2)$ with 
$$x_2=x_2(z):=F^{-1}_r(y_2(\zeta))$$ 
is optimal in \eqref{OS-P_2}. In the rest of this step we prove existence and uniqueness of the solution $y_2(\zeta)\in(\hat{y}_2,+\infty)$ to \eqref{smf00}.

Notice that $\zeta \in (0,\hat{y}_1)$, hence $\zeta < \hat{y}_2$, and by strict concavity of $\hat{G}_2$ on $(0,\hat{y}_2)$ one has
\begin{align}\label{pos1}
\hat{G}'_2(\hat{y}_2)(\hat{y}_2-\zeta)< 
\hat{G}_2(\hat{y}_2) - \hat{G}_2(\zeta).
\end{align}
By substituting the above inequality into \eqref{eq:ell} we get
\begin{align}\label{pos2}
\cL_2(\hat{y}_2,\zeta)> \hat{G}_2(\zeta) - \hat{L}_2(\zeta)>0,
\end{align}
(noting that $\hat{G}_2 > \hat{L}_2$ by $(i)$ in Assumption \ref{ass:flessi}). Also, $u\mapsto\cL_2(u,\zeta)$ is decreasing for $u\in(\hat{y}_2,+\infty)$ since $\frac{\partial}{\partial u}\cL_2(u,\zeta) = -\hat{G}''_2(u)(u-\zeta) <0$ by the convexity of $\hat{G}_2$ and the fact that $\zeta<\hat{y}_1<\hat{y}_2<u$.

Next we show that $\cL_2(u,\zeta)\to -\infty$ as $u\to +\infty$. To this end, note that $\hat{G}_2$ is decreasing on $(\hat{y}_2,+\infty)$ (Lemma \ref{lemm:concavity}), so 
\begin{align}\label{arg}
\lim_{u\to\infty}\big[ \hat{G}_2(u)-\hat{G}'_2(u)(u-\zeta)\big]\le \lim_{u\to\infty}\big[\hat{G}_2(u)-\hat{G}'_2(u)u\big].
\end{align}
Since $\hat{L}_2$ is bounded on $(0,\hat{y}_1)$ (Assumption \ref{ass:flessi}), from \eqref{eq:ell} it is now sufficient to establish that the latter limit equals $-\infty$.

The chain rule and \eqref{res3} (taking $H=G_2$ and $h=g_2$) give
$$\hat{G}'_2(u) u=\frac{u}{F'_r\big(F^{-1}_r(u)\big)}\Big(\frac{G_2}{\phi_r}\Big)'\big(F^{-1}_r(u)\big)=
-\frac{u}{W}\int_{F^{-1}_r(u)}^{\overline{x}}\phi_r(t)g_2(t)m'(t)dt.$$
Setting $u=F_r(s)$, $s \in \cI$, from \eqref{resolvent2} we obtain
\begin{align}
\hat{G}_2(u)-\hat{G}'_2(u)u=-W^{-1}\int_{\underline{x}}^s{\psi_r(t)g_2(t)m'(t)dt}.
\end{align}
Fix $\delta>0$. By the conditions on our stopping costs, the function $g_2$ is bounded below on $[\hat{x}_2+\delta,\overline{x})$ by a constant $\eps_\delta>0$. When $s>\hat{x}_2+\delta$ we split the integral above on the intervals $(\underline{x},\hat{x}_2+\delta]$ and $[\hat{x}_2+\delta,\overline{x})$, then use the bound on $g_2$ and \eqref{psiphiproperties3} to obtain that
\begin{align}\label{eq:lim}
\hat{G}_2(u)-\hat{G}'_2(u)u\le&-W^{-1}\Big[\int_{\underline{x}}^{\hat{x}_2+\delta}\psi_r(t)g_2(t)m'(t)dt+\frac{\eps_\delta}{r}\Big(\frac{\psi'_r(s)}{S'(s)}-\frac{\psi'_r(\hat{x}_2+\delta)}{S'(\hat{x}_2+\delta)}\Big)\Big]
\end{align}
which tends to $-\infty$ as $s \uparrow \overline{x}$ by \eqref{psiphiproperties2}. This  completes Step 1.

As we will see in the proof of Proposition \ref{prop:values}, equation \eqref{smf00} may be interpreted as a geometric version of the so called smooth-fit equation for $P_2$, which specifies that \eqref{OS-P_2} should be continuously differentiable in $x$ across the optimal boundary $x^*_2$. From the arbitrariness of $z\in(\underline{x},\hat{x}_1)$ and a simple application of the implicit function theorem we obtain that the map $z\mapsto x_2(z)$ is continuous on $(\underline{x},\hat{x}_1)$, or equivalently $y_2(\,\cdot\,)\in C((0,\hat{y}_1))$ (see, e.g., Th.\ 10.2.1 on p.\ 270 of \cite{Dieudonne}).
\vspace{+5pt}

\noindent\emph{Step 2 (Player $1$'s best reply)} Similarly suppose that $P_2$ picks $z\in(\hat{x}_2,\overline{x})$ and decides to stop at time $\tau_2(z)$. Then $P_1$ is faced with an optimal stopping problem of the form
\begin{equation}
\label{OS-P_1}
\inf_{\tau \in \mathcal{T}}\EE_x\Big[e^{-r\tau}G_1(X_{\tau})\mathds{1}_{\{\tau < \tau_2(z)\}} + L_1(X_{\tau_2(z)}) e^{-r\tau_2(z)}\mathds{1}_{\{\tau \geq \tau_2(z)\}}\Big].
\end{equation}
It may be proven just as in Step 1 that (with $\zeta:=F_r(z)$) the equation 
\begin{align}\label{smf01}
\cL_1(\,\cdot\,,\zeta)=0
\end{align}
has a unique solution $y_1(\zeta)\in (0,\hat{y}_1)$. Notice by \eqref{smf01} that $y_1(\zeta)>0$ is guaranteed in this setting, by observing that $\hat{G}'_1(0+)=-\infty$ (see Lemma \ref{lemm:concavity}). Then an optimal stopping time for $P_1$ is $\tau_1(x_1)$ where the optimal boundary point is $x_1=x_1(z):=F_r^{-1}(y_1(\zeta))$ (see Appendix \ref{app:secondOS}).

Again, the map $z\mapsto x_1(z)$ is continuous on $(\hat{x}_2,\overline{x})$ (or equivalently $y_1(\,\cdot\,)\in C((\hat{y}_2,+\infty))$) by the implicit function theorem and arbitrariness of $z$.
\vspace{+5pt}

\noindent\emph{Step 3 (A fixed point)}. 
With the continuous functions $x_2(\cdot)$, $x_1(\cdot)$ defined as in Steps 1 and 2 respectively, suppose now that there exist two points $x_1^*,x_2^* \in \cI$ with $x_1^*<x_2^*$ such that
\begin{eqnarray}\label{eq:fpt1}
 x_2^*&=&x_2(x_1^*), \\ x_1^*&=&x_1(x_2^*). \label{eq:fpt2}
\end{eqnarray}
Let us take $\tau_1^*=\tau_1(x^*_1)$, $\tau_2^*=\tau_2(x^*_2)$ and show that they form a Nash equilibrium (Definition \ref{def:Nash}). Since Step 1 constructs $P_2$'s best reponse to $\tau_1^*$ over all stopping times $\tau \in \cT$, we have from \eqref{eq:fpt1} that the lower inequality in \eqref{Nashequilibrium} is satisfied. Similarly Step 2 implies that the upper inequality in \eqref{Nashequilibrium} is satisfied, and so the pair $(\tau_1^*,\tau_2^*)$ is a Nash equilibrium. In this step we will therefore establish the existence of $x_1^*<x_2^*$ satisfying \eqref{eq:fpt1}--\eqref{eq:fpt2}.

We thus seek $y^*_2\in(\hat{y}_2,+\infty)$ such that $\cL_2(y^*_2,y_1(y^*_2))=0$. By the regularity of $G_2$, $L_2$ and $y_1$ we have $u\mapsto\cL_2(u,y_1(u))$ continuous on $(\hat{y}_2,+\infty)$. We conclude just as in \eqref{pos1} and \eqref{pos2} that $\cL_2(\hat{y}_2,y_1(\hat{y}_2))>0$. Since the point $\zeta\in (0,\hat{y}_1)$ in \eqref{arg} does not need to be constant for the latter inequality to hold, the proof of Step 1 also gives that $\lim_{u\uparrow+\infty}\cL_2(u,y_1(u))=-\infty$ and we conclude that $\cL_2(\,\cdot\,,y_1(\,\cdot\,))$ has a root $y^*_2$. From Step 1 we know that $\mathcal{L}_2(\cdot, y_1(y^*_2))=0$ has a unique solution, denoted by $y_2(y_1(y^*_2))$, hence $y_2^*=y_2(y_1(y^*_2))$. Therefore setting $y^*_1:=y_1(y^*_2)$, we have obtained a solution of \eqref{eq:fpt1}-\eqref{eq:fpt2} with  
$x^*_1:=F^{-1}_r(y^*_1)$ and $x^*_2:=F^{-1}_r(y^*_2)$.
\end{proof}

It is worth observing that if $(y_1,y_2)\in(0,\hat{y}_1)\times(\hat{y}_2,+\infty)$ is an arbitrary solution of \eqref{system-construction}, then in particular $y_2$ is the unique solution of $\cL_2(\,\cdot\,,y_1)=0$ in $(\hat{y}_2,+\infty)$ by Step 1 in the proof above. Therefore, recalling \eqref{st-times0} and Step 1, the stopping time $\tau_2(x_2)$ with $x_2:=F^{-1}_r(y_2)$ is optimal in \eqref{OS-P_2} when $z=F^{-1}_r(y_1)=:x_1$. Analogously, from Step 2 we find that $\tau_1(x_1)$ is optimal in \eqref{OS-P_1} when $z=x_2$ and therefore the couple $(\tau_1(x_1),\tau_2(x_2))$ forms a Nash equilibrium. 

Conversely suppose that a couple $(\tau_1(x_1),\tau_2(x_2))$, $x_1<x_2$, forms a Nash equilibrium in the class of threshold type strategies \eqref{st-times0}. Then from Step 1 we have that $y_2:=F_r(x_2)$ must be the solution of $\cL_2(\,\cdot\,,y_1)=0$ in $(\hat{y}_2,+\infty)$, with $y_1:=F_r(x_1)$. Similarly Step 2 implies that $y_1$ solves $\cL_1(\,\cdot\,,y_2)=0$ in $(0,\hat{y}_1)$.

Therefore we have established an equivalence which is summarised in the next corollary.
\begin{coroll}\label{cor:exist}
Let Assumption \ref{ass:flessi} hold. A couple $(\hat{\tau}_1, \hat \tau_2):=(\tau_1(x_1),\tau_2(x_2))$, with $\underline x<x_1<x_2<\overline x$, forms a Nash equilibrium in the class of threshold type strategies \eqref{st-times} if and only if $y_1=F_r(x_1)$ and $y_2=F_r(x_2)$ are a solution of problem \eqref{system-construction}.
\end{coroll}

Next we verify some analytical expressions associated to the equilibrium payoffs constructed above. We apply It\^o's formula but note that direct calculations involving the Laplace transforms of $\tau^*_i$, $i=1,2$ (see Theorem \ref{thm:construction}) and the equilibrium payoffs of the game would also suffice.
\begin{prop}
\label{prop:values}
Let Assumption \ref{ass:flessi} hold and let $(y_1,y_2)\in(0,\hat{y}_1)\times(\hat{y}_2,+\infty)$ be a solution of \eqref{system-construction}. With $x_1,x_2,\hat{\tau}_1,\hat{\tau}_2$ as in Corollary \ref{cor:exist}, $(\hat{\tau}_1,\hat{\tau}_2)$ forms a Nash equilibrium. Moreover the functions
\begin{equation}
\label{eq:value1}
v_1(x):=
\left\{
\begin{array}{ll}
G_1(x), & x \leq x_1,\\[+4pt]
m_1\psi_r(x)+q_1\phi_r(x), &  x_1 < x < x_2,\\[+4pt]
L_1(x), & x \geq x_2,
\end{array}
\right.
\end{equation}
and
\begin{equation}
\label{eq:value2}
v_2(x):=
\left\{
\begin{array}{ll}
L_2(x), & x \leq x_1,\\[+4pt]
m_2\psi_r(x)+q_2\phi_r(x), & x_1 < x < x_2,\\[+4pt]
G_2(x), & x \geq x_2,
\end{array}
\right.
\end{equation}
with
\begin{align}
\label{eqn:lincf1}m_1:=\frac{(G_1/\phi_r)(x_1)-(L_1/\phi_r)(x_2)}{F_r(x_1)-F_r(x_2)},\qquad
q_1:=\frac{L_1}{\phi_r}(x_2)-m_1 F_r(x_2),\\[+4pt]
m_2:=\frac{(G_2/\phi_r)(x_2)-(L_2/\phi_r)(x_1)}{F_r(x_2)-F_r(x_1)},\qquad
q_2:=\frac{L_2}{\phi_r}(x_1)-m_2 F_r(x_1),
\end{align}
coincide with the equilibrium payoffs of the two players, i.e.~$v_i(x)=\cJ_i(\hat{\tau}_1,\hat{\tau}_2;x)$, $i=1,2$. In particular $v_1\in C({\cI})$ with $v_1\in W^{2,\infty}_{loc}(\underline{x},x_2)$, $v_2\in C({\cI})$ with $v_2\in W^{2,\infty}_{loc}(x_1,\overline{x})$ and they solve 
\begin{align}
\label{freeb1}&(\LL_X-r)v_i(x) =0, & x_1<x<x_2, \:i=1,2\\[+3pt]
\label{freeb2}&(\LL_X-r)v_1(x) >0, & \underline{x}<x<x_1\\[+3pt]
\label{freeb3}&(\LL_X-r)v_2(x) >0, & x_2<x<\overline{x}\\[+3pt]
\label{freeb4}&v_i\le G_i,& x\in\cI,\:i=1,2.
\end{align}
\end{prop}
\begin{proof}
The fact that $(\hat{\tau}_1,\hat{\tau}_2)$ defines a Nash equilibrium follows from Corollary \ref{cor:exist}. The rest of the proof is organised in three steps.
\vspace{+5pt}

\noindent\emph{Step 1 (Regularity of $v_i$)}. At this point we recall the smooth change of variables defined in \eqref{def:hat}, writing  $x=F^{-1}_r(y)$. Applying this change of variables to $v_1$ in \eqref{eq:value1}, the function $y \mapsto \hat{v}_1(y)$ is a \emph{straight line} on $(y_1,y_2)$. The coefficients of this straight line, given in \eqref{eqn:lincf1}, ensure that $\hat{v}_1$ is continuous on $(0,\infty)$, and hence that $v_1 \in C({\cI})$.

Further by the definition of $\cL_1$ in \eqref{eq:ell}, it follows from the system \eqref{system-construction} that the gradient of this straight line is equal to the derivative of $\hat{G}_1$ at $y_1$. We conclude that $\hat{v}_1$ is continuously differentiable at $y_1$ or, equivalently, that $v_1$ is continuously differentiable at $x_1$. In this sense equation \eqref{smf01} is a geometric version of the smooth fit equation for $v_1$. It follows immediately that $v_1$ is continuously differentiable on $(\underline{x},x_2)$. 
Similarly by direct calculations and \eqref{eq:value1}-\eqref{eq:value2} we can check that $v''_1$ is indeed a locally bounded function on $(\underline{x},x_2)$, hence $v_1\in W^{2,\infty}_{loc}(\underline{x},x_2)$.  We can proceed in a similar way for $v_2$.
\vspace{+5pt}

\noindent\emph{Step 2 (Free boundary problem)}. The equations \eqref{freeb1}, \eqref{freeb2} and \eqref{freeb3} follow directly from the definition of $v_i$, by recalling that $\phi_r$ and $\psi_r$ solve $\LL_Xu-ru=0$, and by the fact that $x_1<\hat{x}_1$ and $x_2>\hat{x}_2$. For the final inequalities (the so-called {\em obstacle conditions}) we refer again to the transformation  \eqref{def:hat}. The transformed function $\hat{G}_1$ is convex in $(0,\hat{y}_1)$, it reaches its unique global minimum therein and it is concave in $(\hat{y}_1,+\infty)$. By the smooth fit property established above at $y_1 \in (0,\hat{y}_1)$, it follows from $\hat{v}_1(y_2)=\hat{L}_1(y_2)<\hat{G}_1(y_2)$ that we must also have $\hat{v}_1\le \hat{G}_1$ on $(y_1,y_2)$. Therefore we have $v_1\le G_1$ in $(\underline{x},x_2)$ and $v_1=L_1<G_1$ in $[x_2,\overline{x})$. Symmetric arguments hold for $v_2$.

\vspace{+5pt}

\noindent\emph{Step 3 (Verification argument)}. Here we show that indeed $v_i$, $i=1,2$ coincide with the equilibrium payoffs. As a byproduct, this step offers an alternative way of showing that $(\hat{\tau}_1,\hat{\tau}_2)$ is a Nash equilibrium, starting from the solution of \eqref{freeb1}--\eqref{freeb4} (this is the original approach of \cite{BensoussanFriedman}).

Let $\sigma\in\cT$ be arbitrary. Then we have
\begin{align}\label{dyn00}
v_1(x)=&\EE_x\Big[e^{-r(\sigma\wedge\hat{\tau}_2)}v_1(X_{\sigma\wedge\hat{\tau}_2})-\int_0^{\sigma\wedge\hat{\tau}_2}e^{-rt}(\LL_X-r)v_1(X_t)dt\Big]
\nonumber\\[+3pt]
\le& \EE_x\Big[e^{-r\sigma}G_1(X^x_{\sigma})\mathds{1}_{\{\sigma<\hat{\tau}_2\}}+e^{-r\hat{\tau}_2}L_1(X_{\hat{\tau}_2})\mathds{1}_
{\{\sigma\ge\hat{\tau}_2\}}\Big]=\cJ_1(\sigma,\hat{\tau}_2;x).
\end{align}
Here the first line follows from It\^o-Tanaka's formula (justified by the regularity of $v_1$) and a standard localisation argument, and the second line follows from \eqref{eq:value1}, \eqref{freeb1}, \eqref{freeb2} and \eqref{freeb4}. In particular, setting $\sigma=\hat{\tau}_1$ in \eqref{dyn00} we obtain $v_1(x)=\cJ_1(\hat{\tau}_1,\hat{\tau}_2;x)$. Arguing similarly for $v_2$ yields the claimed equivalence of $v_i$, $i=1,2$ with the equilibrium payoffs.
\end{proof}

The application of It\^o's formula in Step 3 of the latter proof also yields the following (sub)-martingale property of the processes $t\mapsto e^{-rt}v_i(X_t)$, $i=1,2$. This is the analogue in our game setting of the well established (sub)-martingale property in optimal stopping problems with minimisation over stopping times (see, e.g., \cite{Pham}, Chapter 5, Section 2.3). 
\begin{coroll}\label{cor:supsubarm}
Let $(\hat{\tau}_1,\hat{\tau}_2)$ be as in Proposition \ref{prop:values} and $v_i$, $i=1,2$ the related equilibrium payoffs for the two players. For $i,j=1,2$ and $i\neq j$ set
\begin{align}
Y^i_t:=e^{-rt}v_i(X_{t}),\:\:\:G^i_t:=e^{-rt}G_i(X_{t}),\:\:\:
L^i_t:=e^{-rt}L_i(X_{t}),\quad t\ge0
\end{align}
then $(Y^i_{t\wedge\hat{\tau}_j})_{t\ge0}$ is a continuous sub-martingale, $(Y^i_{t\wedge\hat{\tau}_i\wedge\hat{\tau}_j})_{t\ge0}$ is a continuous martingale, $Y^i_t\le G^i_t$ for all $t\ge0$ and $Y^i_{\hat{\tau}_i\wedge\hat{\tau}_j}=G^i_{\hat{\tau}_i}\mathds{1}_{\{\hat{\tau}_i<\hat{\tau}_j\}}+L^i_{\hat{\tau}_j}\mathds{1}_{\{\hat{\tau}_i
>\hat{\tau}_j\}}$ (notice that $\PP_x(\hat \tau_1=\hat \tau_2)=0$ for all $x\in\cI$).
\end{coroll}

Our nonzero-sum game may have multiple Nash equilibria but we now provide sufficient conditions under which the equilibrium of Theorem \ref{thm:construction} is unique in the class \eqref{st-times}.
For this we will consider the auxiliary problem
\begin{align}
\label{eq:auxOS}
\inf_{\tau\in\cT}\EE_x\big[e^{-r\tau}G_2(X_\tau)\big],\qquad x\in\cI,
\end{align}
which corresponds to the optimal stopping problem for $P_2$ if $P_1$ decides never to stop. The proof of the next lemma is standard and we provide it in Appendix \ref{app:lem}. 
\begin{lemma}\label{lem:aux1}
There is a unique solution $y_2^\infty$ in $(\hat{y}_2,+\infty)$ of  
$\hat{G}'_2(y)y-\hat{G}_2(y)=0$. Setting $x_2^\infty:=F^{-1}_r(y_2^\infty)>\hat{x}_2$ and recalling 
\eqref{st-times0}, the stopping time $\tau^\infty_2:=\tau_2(x^\infty_2)$  
is an optimal stopping time for \eqref{eq:auxOS}.
\end{lemma}

Let $y_1^\infty$ be the unique $y\in(0,\hat{y}_1)$ that solves $\cL_1(\,\cdot\,,y_2^\infty)=0$, whose existence we know from Step 2 in the proof of Theorem \ref{thm:construction}. Then the latter arguments also give us the next corollary.
\begin{coroll}\label{cor:y1}
Set $x_1^\infty:=F^{-1}_r(y_1^\infty)<\hat{x}_1$. Then $\tau^\infty_1:=\tau_1(x^\infty_1)$ (see \eqref{st-times0}) provides the best reply of $P_1$ when $P_2$ stops at $\tau^\infty_2$.
\end{coroll}
\noindent For future reference it is worth recalling that the optimal stopping problem for $P_1$ when $P_2$ stops at $\tau^\infty_2$ is
\begin{align}\label{eq:auxOS2}
\inf_{\tau\in\cT}\EE_x\Big[ e^{-r\tau}G_1(X_\tau)\mathds{1}_{\{\tau<\tau^\infty_2\}}+e^{-r\tau^\infty_2}L_1(X_{\tau_2^\infty})\mathds{1}_{\{\tau\ge\tau^\infty_2\}}\Big],\qquad x\in\cI.
\end{align}

Recalling $\widetilde{y}_i$, $i=1,2$, from Definition \ref{def:inflection} we are now ready to state our uniqueness result.
\begin{theorem}[\textbf{Uniqueness of the equilibrium}]\label{thm:unique}
Let Assumption \ref{ass:flessi} hold and let us also assume 
\begin{itemize}
\item[(i)] $L_i\in\cA_i$, $i=1,2$
\item[(ii)] $\widetilde{y}_2>\hat{y}_1$ 
\item[(iii)] $\hat{G}'_1(y^\infty_1)<\hat{L}'_1(y^\infty_2)$
\end{itemize}
Then problem \eqref{system-construction} has a unique solution. Writing $x^*_i=F^{-1}_r(y^*_i)$, $i=1,2$, then $(x^*_1,x^*_2)$ is the unique couple such that $(\tau_1(x^*_1),\tau_2(x^*_2))$ as in \eqref{st-times} constitutes a Nash equilibrium for the game.
\end{theorem}
\begin{proof}
The main idea of the proof is to show that the functions $\zeta\mapsto y_i(\zeta)$ for $i=1,2$ found in the proof of Theorem \ref{thm:construction} are monotonic, respectively increasing and decreasing, so that they intersect at most once and \eqref{eq:fpt1}-\eqref{eq:fpt2} has a unique solution. Uniqueness of the equilibrium (in the class  \eqref{st-times}) then follows by Corollary \ref{cor:exist}. We adopt the notation of Theorem \ref{thm:construction} and observe that under the additional regularity assumptions on $L_i$ the implicit function theorem implies that $y_i(\,\cdot\,)\in C^1(\cO_i)$ with $i=1,2$ and $\cO_1:=(\hat{y}_2,+\infty)$, $\cO_2:=(0,\hat{y}_1)$. In fact denoting by $\partial_k\cL_i$ the partial derivative of $\cL_i$ with respect to the $k$-th variable $k=1,2$, the implicit function theorem gives
\begin{align}\label{impl}
y'_i(\zeta)=-\frac{\partial_2\cL_i}{\partial_1\cL_i}(y_i(\zeta),\zeta)=\frac{\hat{G}'_i(y_i(\zeta))-\hat{L}'_i(\zeta)}{\hat{G}''_i(y_i(\zeta))
(y_i(\zeta)-\zeta)},\quad \zeta\in\cO_i,\:i=1,2.
\end{align}
\vspace{+4pt}

\noindent\emph{Step 1 ($y_2$ is monotonic decreasing)}. First we want to prove that $y_2(\,\cdot\,)$ decreases monotonically on $\cO_2$. 
For $\zeta\in\cO_2$ it holds $y_2(\zeta)>\hat{y}_2>\zeta$. Hence by Lemma \ref{lemm:concavity}-$(i)$ we have $\hat{G}''_2(y_2(\zeta))(y_2(\zeta)-\zeta)>0$  because $\hat{G}_2$ is convex, and $\hat{G}'_2(y_2(\zeta))<0$ by Lemma \ref{lemm:concavity}-$(iii)$. By assumption $(ii)$ we also have $\hat{L}'_2(\zeta)>0$ for $\zeta\in\cO_2$, since $\hat{L}_2$ is increasing on $(0,\widetilde{y}_2)$ by Lemma \ref{lemm:concavity}. Therefore we have from \eqref{impl} that $y_2$ is decreasing on $\cO_2$ as claimed. 
From Lemma \ref{lem:aux1} we find the maximum value $y_2(0+)=y_2^\infty$. 

Since the optimal boundaries constructed in Theorem \ref{thm:construction} have the fixed point property that $y^*_1=y_1(y_2(y^*_1))$ (cf.~\eqref{eq:fpt1}--\eqref{eq:fpt2}), it is sufficient to show monotonicity of $\zeta\mapsto y_1(\zeta)$ on the interval $\zeta\in(\hat{y}_2, y^\infty_2)$, which contains the range of $y_2(\,\cdot\,)$. This is done in the next step.
\vspace{+4pt}

\noindent\emph{Step 2 ($y_1$ is monotonic increasing)}.
Taking $\zeta\in\cO_1$ we have $y_1(\zeta)\in\cO_2$ and by $(i)$ of Lemma \ref{lemm:concavity} we have $\hat{G}''_1(y_1(\zeta))(y_1(\zeta)-\zeta)<0$. 
Corollary \ref{cor:y1} and Step 2 in the proof of Theorem \ref{thm:construction} justify setting $y^\infty_1=y_1(y^\infty_2)$. Hence we can write $\hat{G}'_1(y_1(y^\infty_2))= \hat{G}'_1(y^\infty_1)$ and since we are assuming $\hat{G}'_1(y^\infty_1)<\hat{L}'_1(y^\infty_2)$, then
\begin{align}\label{pos}
y'_1(y^\infty_2)>0.
\end{align}

Let us now study the sign of the function $U:(\hat{y}_2,y^\infty_2) \to \R$,
where $U(\zeta):= \hat{G}'_1(y_1(\zeta))-\hat{L}'_1(\zeta)$. 
Assume that $U$ has a zero at $\zeta^o_1$ or, equivalently, that $y'_1(\zeta^o_1)=0$. Then since $\cL_1(y_1(\zeta^o_1),\zeta^o_1)=0$ (cf. \eqref{eq:ell}) and $U(\zeta^o_1)=0$, there is a straight line which is tangent both to $\hat{L}_1$, at $\zeta^o_1$, and to $\hat{G}_1$, at $y_1(\zeta^o_1)$. Since $\hat{L}_1$ is convex for $y<\check{y}_1$
and $\hat{L}_1<\hat{G}_1$ it is easy to see that we must have $\zeta^o_1>\check{y}_1$, otherwise the tangent would lie below $\hat{L}_1$ on $(0,\zeta^o_1)$ and violate $\hat{L}_1<\hat{G}_1$ at $y_1(\zeta^o_1)$. 

Now we claim that if such $\zeta^o_1$ exists, then 
\begin{align}
\label{monot01}
&U>0\:\:\:\text{and}\:\:\:y'_1<0\quad \text{on $(\zeta^o_1, y^\infty_2)$}.
\end{align}
Since the latter inequality would contradict \eqref{pos}, it would then follow that $y'_1(\cdot)$ must be strictly positive on $(\hat{y}_2,y^\infty_2)$.

Hence, to conclude it remains to prove \eqref{monot01}. For this we use $y'_1(\zeta^o_1)=0$ and observe that  
\begin{equation}
\label{sign-uniq}
U'(\zeta^o_1) = \Big(\hat{G}^{''}_1(y_1(\zeta))y^{'}_1(\zeta)-\hat{L}^{''}_1(\zeta)\Big)\Big|_{\zeta=\zeta^o_1} = -\hat{L}^{''}_1(\zeta^o_1) > 0,
\end{equation}
since $\zeta^o_1>\check{y}_1$.  Hence $U$ may only equal zero with strictly positive derivative, so it has at most one zero $\zeta^o_1$ and then \eqref{monot01} holds. 
\vspace{+4pt}

\noindent\emph{Step 3 (Uniqueness of the solution to \eqref{system-construction})}. 
From Theorem \ref{thm:construction} we know that there exists $(y_1^*, y_2^*) \in \cO_2 \times \cO_1$ such that $y_2^*=y_2(y_1^*)$ and $y_1^*=y_1(y_2^*)$. The monotonicity of both $y_1$ and $y_2$ obtained above implies that this pair is unique in $\cO_2 \times \cO_1$, and hence there is a unique solution to \eqref{system-construction}.
\end{proof}


\subsection{The case of an entrance-not-exit boundary}
\label{entrance}

In this section we extend the methodology developed above to the case when $\underline{x}$ is an entrance-not-exit boundary and $\overline{x}$ is a natural boundary for $X$. This setting includes for example certain CIR and Bessel processes 
(see for instance \cite{JYC}). For the fundamental solutions $\phi_r$ and $\psi_r$ we have that \eqref{psiphiproperties1}, \eqref{psiphiproperties2}, \eqref{psiphiproperties3} continue to hold if we replace (see \cite{BS}, Sec.~2, par.~10)
\begin{align}
\label{newpsi}& \lim_{x\downarrow \underline{x}}\psi_r(x)=0\:\:\:\text{by}\:\:\:\lim_{x\downarrow \underline{x}}\psi_r(x)>0,\\
\label{newphi}& \lim_{x\downarrow \underline{x}}\frac{\phi'_r(x)}{S'(x)}=-\infty\:\:\:\text{by}\:\:\:\lim_{x\downarrow \underline{x}}\frac{\phi'_r(x)}{S'(x)}>-\infty.
\end{align}

This setting is adopted in the remainder of this section. 
We first examine the geometric properties associated with the classes $\cA_1$, $\cA_2$, as was done previously in Lemma \ref{lemm:concavity}, under the new boundary behaviour for $\underline x$. The 
asymptotic behaviour of $\hat H$ as $y\uparrow +\infty$ is 
exactly the same 
as in Lemma \ref{lemm:concavity} since the upper endpoint of $\cI$ is again natural. Notice as well that $\hat H(0+)=0$ by definition of $\cA_i$ and that $\hat H'(0+)$ always exists by convexity or concavity. Compared to Lemma \ref{lemm:concavity}, the difference in the present setting is that functions in $\cA_1$ and $\cA_2$ may now have finite derivative at zero with either negative or positive sign.

\begin{lemma}
\label{lemm:concavity2}
\ 
\begin{itemize}
	\item[  i)] If $H\in\cA_1$ then $\hat{H}$ is convex on $(0, \hat{y}_h)$ and concave on $(\hat{y}_h, \infty)$. Moreover $\hat{H}(0+)=0$, $\lim_{y\to\infty}\hat{H}(y)=+\infty$ and $H$ is monotonic increasing on $(\hat{y}_h,+\infty)$. 
	
In addition if $\hat{H}'(0+)<0$ then $\hat{H}$ has a unique global minimum in $(0, \hat{y}_h)$.

	\item[ii)] If $H\in\cA_2$ then $\hat{H}$ is concave on $(0, \hat{y}_h)$ and convex on $(\hat{y}_h, \infty)$. Moreover $\hat{H}(0+)=0$, $\lim_{y\to\infty}\hat{H}(y)=-\infty$ and $\hat{H}$ is monotonic decreasing on $(\hat{y}_h,+\infty)$. 
	
In addition if $\hat{H}'(0+)> 0$ then $\hat{H}$ has a unique global maximum in $(0, \hat{y}_h)$.  
\end{itemize}
\end{lemma}
\noindent Notice that by $(ii)$ of Lemma \ref{lemm:concavity2}, Lemma \ref{lem:aux1} continues to hold. 
We now examine the effect of the modified geometry on the remaining results.

In the case that $\hat{G}'_1(0+)=-\infty$, all geometrical considerations are identical to those of the setting of Section \ref{Nash-natural} and so the proof of the next result is the same as that of Theorem \ref{thm:construction}.
\begin{prop}
If $\hat{G}'_1(0+)=-\infty$ then Theorem \ref{thm:construction} continues to hold when $\underline{x}$ is an entrance-not-exit boundary. 
\end{prop}
Next we analyse cases in which $\hat{G}'_1(0+)\in(-\infty,0)$. Firstly we establish the existence of equilibria having a degenerate version of the threshold type structure \eqref{st-times}, in the sense that we formally take $x_1^*=\underline x$.
\begin{prop}\label{prop:inf}
Let Assumption \ref{ass:flessi} hold, let $\hat{G}'_1(0+)<0$ and recall $\tau^\infty_2:=\tau_2(x^\infty_2)$ from Lemma \ref{lem:aux1}. Then $(+\infty,\tau^\infty_2)$ is a Nash equilibrium if and only if
\begin{align}\label{ineq}
\frac{\hat{L}_1(y^\infty_2)}{y^\infty_2}\le\,\hat{G}'_1(0+) 
\end{align}
with $y^\infty_2=F_r(x^\infty_2)$.
\end{prop}
\begin{proof}
\noindent \emph{Step 1 (Sufficiency)}. Suppose \eqref{ineq} holds and let $P_2$ choose the stopping time $\tau^\infty_2$ which is optimal in problem \eqref{eq:auxOS}, so that $P_1$ is faced with solving \eqref{eq:auxOS2}. Due to condition \eqref{ineq}, the largest convex function $W_1$ dominated by $\hat{G}_1$ on $[0,y^\infty_2]$ such that $W_1(y)=\hat{L}_1(y)$, for $y \geq y^\infty_2$, describes the value function (see details in Appendix \ref{2usefulOS}). This $W_1$ is given by the straight line starting from the origin and passing through $(y^\infty_2,\hat{L}_1(y^\infty_2))$. Therefore due to strict convexity of $\hat{G}_1$ at zero, $P_1$'s best reply to $\tau^\infty_2$ is the stopping time $\tau_1(\underline{x})=\inf\{t\ge0\,:\,X_t=\underline{x}\}=+\infty$ a.s.~(since the entrance-not-exit boundary $\underline{x}$ is unattainable in finite time). Since $\tau^\infty_2$ is also $P_2$'s best reply to $\tau_1(\underline{x})$ we have a Nash equilibrium.
\vspace{+4pt}

\noindent \emph{Step 2 (Necessity)}. We show necessity by contradiction. Suppose that $(+\infty,\tau^\infty_2)$ is a Nash equilibrium and that \eqref{ineq} does not hold. 

Let $P_2$ choose the stopping time $\tau^\infty_2$ so that $P_1$ must solve \eqref{eq:auxOS2}. Since \eqref{ineq} does not hold it is not possible to draw a straight line joining the origin to $(y^\infty_2, \hat{L}_1(y^\infty_2))$ and lying below $\hat{G}_1$ on $(0,y^\infty_2)$. This line would be $P_1$'s payoff for never stopping, therefore $\tau_1=+\infty$ cannot be a best reply.
\end{proof}

The above proposition shows that the construction of Theorem \ref{thm:construction} may break down in some cases, due to the geometry of $\hat G_1$. Hence in our present setting establishing the {\em existence} of an equilibrium requires different conditions on the cost functions, such as those in the next proposition (and indeed there may be cases where no equilibrium can be found in our class of strategies \eqref{st-times}). 

\begin{prop}\label{prop:L1hatlim}
Let
\begin{align}\label{ineq2}
-\infty<\hat{G}'_1(0+)< \frac{\hat{L}_1(y^\infty_2)}{y^\infty_2}<0
\end{align}
and assume $\lim_{y\to\infty}\hat{L}_1(y)>-\infty$ and $y_T\le \hat{y}_2$, where
\begin{align}\label{yT}
y_T:=\sup\{y>0\,:\,\hat{G}_1'(0+)y=\hat{L}_1(y)\},\quad \text{with} \quad \sup\{\emptyset\} = 0,
\end{align}
then Theorem \ref{thm:construction} continues to hold when $\underline{x}$ is an entrance-not-exit boundary.
\end{prop}
\begin{proof}

Since $\hat{G}'_1(0+)y$ does not cross $\hat{L}_1(y)$ for $y>y_T$, it must be that $\hat{L}_1(y)-\hat{G}'_1(0+)y$ is either strictly positive or strictly negative for $y>y_T$.  However the latter would violate \eqref{ineq2}, since $y^\infty_2>\hat{y}_2\ge y_T$, and hence is impossible. Then we must have $\hat{L}_1(y)-\hat{G}'_1(0+)y>0$ for $y>\hat{y}_2$. 
Hence by strict convexity of $\hat{G}_1$ in  $(0,\hat{y}_1)$ and a simple geometric argument, for $\zeta>\hat y_2$ one can always construct a unique straight line passing through $(\zeta,\hat{L}_1(\zeta))$ and tangent to $\hat{G}_1$ at a point of $(0,\hat{y}_1)$. Thus $\cL_1(\,\cdot\,,\zeta)$ has a unique root $y_1(\zeta)\in(0,\hat{y}_1)$ for each $\zeta>\hat{y}_2$.  

This argument shows that Step 2 of the proof of Theorem \ref{thm:construction} may be carried out in the present setting. Step $1$ is analogous, and Step $3$ follows.
\end{proof}

Under the assumptions of Proposition \ref{prop:L1hatlim}, we notice that Corollary \ref{cor:exist}  and Proposition \ref{prop:values} continue to hold, i.e.~any solution of \eqref{system-construction} leads to a Nash equilibrium of threshold type and to the related analytical properties of the equilibrium payoffs.

\begin{remark}
1. It is important to notice that for the existence of an equilibrium we have not examined whether or not $\hat{L}_2$ and $\hat{G}_2$ have maxima (see the proof of Theorem \ref{thm:construction}). Instead the existence of these maxima and their position is used in Theorem \ref{thm:unique} to establish uniqueness of the equilibrium. In the current setting $\hat{L}_2$ and $\hat{G}_2$ have maxima if and only if $\hat{L}'_2(0+)>0$ and $\hat{G}'_2(0+)>0$. Therefore assuming the latter along with conditions of Proposition \ref{prop:L1hatlim} we have that Theorem \ref{thm:unique} holds.

2. Even though $\phi_r(x)\uparrow+\infty$ as $x\downarrow \underline{x}$, when $\underline{x}$ is an entrance-not-exit boundary condition \eqref{lims} may become more restrictive. For instance for a Bessel process with index $\nu=1/2$ (i.e.~dimension $\delta=3$) one has $\phi_r(x)\sim 1/x$ as $x\to0$ (see \cite{BS} Appendix 2, pp.~638 and 654). In this case, we may
relax \eqref{lims} for $G_1$ by requiring
\begin{align*}
\lim_{x\downarrow \underline{x}}\frac{G_1}{\phi_r}(x)=A_{G_1}\in(-\infty,+\infty).
\end{align*}
All the above arguments can then be adapted to establish the existence and uniqueness results for Nash equilibria. We omit further details here because in the next section we analyse a similar situation in the case when $\underline{x}$ is an exit-not-entrance boundary and \eqref{lims} becomes a serious restriction.
\end{remark}

\subsection{The case of an exit-not-entrance boundary}
\label{exit}

Here we extend the analysis carried out in the previous two sections by addressing the case of a diffusion with a lower exit-not-entrance boundary $\underline{x}$ and an upper natural boundary $\overline{x}$. We sketch most proofs, drawing out key differences with the previous arguments.

Equations \eqref{psiphiproperties1}, \eqref{psiphiproperties2} and \eqref{psiphiproperties3} continue to hold if we replace
\begin{align}
\label{newphi2}& \lim_{x\downarrow \underline{x}}\phi_r(x)=+\infty\:\:\:\text{by}\:\:\:\lim_{x\downarrow \underline{x}}\phi_r(x)<+\infty,\\
\label{newpsi2}& \lim_{x\downarrow \underline{x}}\frac{\psi'_r(x)}{S'(x)}=0\:\:\:\text{by}\:\:\:\lim_{x\downarrow \underline{x}}\frac{\psi'_r(x)}{S'(x)}>0.
\end{align}
This setting is adopted in the remainder of this section.

We see that $\phi_r(\underline{x}+)$ is now finite so that imposing \eqref{lims} on the stopping costs requires them to vanish at $\underline{x}$ (recall that $\phi_r$ is positive). Hence from now on we shall relax the definition of the set $\cA$ by replacing the condition \eqref{lims} with
\begin{align}\label{repl}
\lim_{x\downarrow\underline{x}}\frac{H}{\phi_r}(x)=A_H
\end{align}
for some $A_H\in\RR$ depending on $H$. 
For any $H\in\cA$ Dynkin's formula, standard localisation and \eqref{resolvent} give
\begin{align}\label{res2a}
H(x)=A_H\phi_r(x)-W^{-1}\Big[\phi_r(x)\int^x_{\underline{x}}{\psi_r(y)h(y)m'(y)dy}+\psi_r(x)\int_x^{\overline{x}}{\phi_r(y)h(y)m'(y)dy}\Big]
\end{align}
and for $(H/\phi_r)'(x)$ we have the same expression as in \eqref{res3}. 

The geometric implications of the present setting are as follows. Since $\lim_{x\downarrow\underline{x}}(\phi'_r/S')(x) = - \infty$ as in the natural boundary case, one can prove as in Lemma \ref{lemm:concavity} that
\begin{align}\label{infd}
H\in\cA_1\Rightarrow \hat{H}'(0+)=-\infty\quad\text{and}\quad H\in\cA_2\Rightarrow \hat{H}'(0+)=+\infty.
\end{align}

Thanks to the latter observation one has that, under the new definition of $\cA$, Lemma \ref{lemm:concavity} holds for 
in the same form with only the exception of the lower boundary conditions: now indeed we have $\hat{H}(0+)=A_{H}$.
As one may expect the sign of $A_{G_1}$ plays a crucial role in determining the existence of Nash equilibria. We study the two possible cases below, while we always assume $A_{G_2}\ge 0$ for simplicity.
\begin{prop}\label{prop:exit1}
If $A_{G_1}\le 0$ then Theorem \ref{thm:construction} holds when $\underline{x}$ is an exit-not-entrance boundary.
\end{prop}
\begin{proof}
Condition \eqref{infd} implies that the construction of an equilibrium follows as in the proof of Theorem \ref{thm:construction} up to trivial adjustments.
\end{proof}

We now consider $A_{G_1}>0$ but with the additional requirement 
\begin{align}\label{eq:inf}
\inf_{x\in\cI}G_1(x)<0.
\end{align}

In this case, from the above mentioned geometry of $\hat G_i$ there exists a unique straight line passing through the origin and tangent to $\hat{G}_1$. We denote the tangency point by $(y_S,\hat{G}_1(y_S))$ so that $y_S\in(0,\hat y_1)$ is the unique solution of
\begin{align}
\hat{G}_1(y_S)={y_S}\hat{G}'_1(y_S).
\end{align}
Repeating arguments as in the proof of Proposition \ref{prop:inf}, up to straightforward modifications, we obtain a similar result:
\begin{prop}
Let $A_{G_1}>0$ and assume \eqref{eq:inf}. Let Assumption \ref{ass:flessi} hold with \eqref{repl} in place of \eqref{lims}. Let also $\tau^\infty_2=\tau_2(x^\infty_2)$ be optimal for \eqref{eq:auxOS}. Then $(\sigma_{\cI},\tau^\infty_2)$ is a Nash equilibrium if and only if
\begin{align}\label{ineq3}
\hat{G}'_1(y_S)> \frac{\hat{L}_1(y^\infty_2)}{y^\infty_2}
\end{align}
with $y^\infty_2=F_r(x^\infty_2)$.
\end{prop}

We now introduce 
\begin{align}
\hat{y}_T:=\sup\{y\ge y_S,\:\,\hat{G}'_1(y_S)y=\hat{L}_1(y)\}, \quad \text{with} \quad \sup\{\emptyset\}=y_S,
\end{align}
which will play a similar role to $y_T$ in the previous section. Before stating the next result we recall that since $\hat{G}_2$ is concave to the left of $\hat x_2$, an optimal boundary for $P_2$ will never lie in $(\underline x,\hat{x}_2)$ (see Step 1 in the proof of Theorem \ref{thm:construction}). 
\begin{prop}\label{prop:S}
Assume that $A_{G_1}>0$, that \eqref{eq:inf} holds and
\begin{align}\label{ineq4}
\hat{G}'_1(y_S)\le \frac{\hat{L}_1(y^\infty_2)}{y^\infty_2}<0.
\end{align}
Assume also that $\hat{y}_T<\hat{y}_2$ and $\lim_{y\to\infty}\hat{L}_1(y)>-\infty$. Set $x_S:=F^{-1}_r(y_S)$ and $\sigma_S:=\inf\{t\ge 0\,:\,X_t\ge x_S\}\wedge\sigma_\cI$, then with \eqref{repl} in place of \eqref{lims} in Assumption \ref{ass:flessi} one has
\begin{itemize}
\item[ a)] the couple $(\sigma_S, +\infty)$ is a Nash equilibrium for the game started at $x\in(\underline{x},x_S]$;
\item[ b)] the couple $(\tau^*_1,\tau^*_2)$ of Theorem \ref{thm:construction} is a Nash equilibrium for the game started at $x>x_S$.
\end{itemize}
\end{prop}
\begin{proof}
We omit details of the proof which involve the repetition of arguments employed several times above. In terms of the problem geometry, the only difference in the present case is that for any $y_0>\hat{y}_2$ the largest convex function $W_1$ dominated by $\hat{G}_1$ and passing through $(y_0,\hat{L}_1(y_0))$ has at most two straight portions: $i)$ the usual one connecting $\hat{L}_1(y_0)$ to $\hat{G}_1$ via the smooth-fit equation $\cL_1(y_1(y_0),y_0)=0$ and $ii)$ the straight line $r_S(y):=\hat{G}'_1(y_S)\,y$ for $y\in[0,y_S]$. Proposition \ref{prop:app-OS-tris} shows that $W_1$ provides $P_1$'s minimal expected cost in this setting.

If $x\in(\underline{x},x_S]$ then with probability one $P_1$ stops prior to $P_2$, at time $\sigma_S$, because $\hat{G}_2$ is concave on $[0,y_S]$. Hence $a)$ holds, because stopping in finite time can only increase $P_2$'s expected cost. On the other hand $b)$ is obtained as in the proof of Theorem \ref{thm:construction}.
\end{proof}

In the setting of Proposition \ref{prop:S}, for each $x$ the additional assumptions of Theorem \ref{thm:unique} are again sufficient for the uniqueness of the equilibria we have obtained. Similarly they are also sufficient in the case $A_{G_1}\le0$. 

We also remark that for $A_{G_1}\le 0$, Corollary \ref{cor:exist} and Proposition \ref{prop:values} hold in the same form whereas for $A_{G_1}>0$ they hold in a slightly more complex form. We provide a full statement for completeness but skip the proof as it is the same as the original one up to minor adjustments.
\begin{prop}
\label{prop:values2}
Let all the assumptions of Proposition \ref{prop:S} hold. Let $(y_1,y_2)\in(0,\hat{y}_1)\times(\hat{y}_2,+\infty)$ be a solution of \eqref{system-construction} and for $x_i:=F^{-1}_r(y_i)$, $i=1,2$ set
\begin{align}\label{eq:ver00}
\hat{\tau}_1:=\tau_1(x_1),\quad\hat{\tau}_2:=\tau_2(x_2).
\end{align}
Then the couple $(\hat{\tau}_1, \hat{\tau}_2)$ is a Nash equilibrium for the game started at $x\in[x_S, \overline{x})$  whereas the couple $(\sigma_S,+\infty)$ is a Nash equilibrium for the game started at $x\in(\underline{x},x_S)$.

Moreover the equilibrium payoffs of the two players, i.e.~$v_i(x)=\cJ_i(\hat{\tau}_1,\hat{\tau}_2;x)$, $i=1,2$, are given by the functions
\begin{equation}
\label{eq:value10}
v_1(x):=
\left\{
\begin{array}{ll}
p_1\psi_r(x), &\underline{x}<x<x_S, \\[+4pt]
G_1(x), & x_S\le x \leq x_1,\\[+4pt]
m_1\psi_r(x)+q_1\phi_r(x) &  x_1 < x < x_2,\\[+4pt]
L_1(x), & x \geq x_2,
\end{array}
\right.
\end{equation}
and
\begin{equation}
\label{eq:value20}
v_2(x):=
\left\{
\begin{array}{ll}
p_2\psi_r(x), &\underline{x}<x<x_S, \\[+4pt]
L_2(x), & x_S\le x \leq x_1,\\[+4pt]
m_2\psi_r(x)+q_2\phi_r(x), & x_1 < x < x_2\\[+4pt]
G_2(x), & x \geq x_2,
\end{array}
\right.
\end{equation}
with $m_i$, $q_i$, $i=1,2$ as in Proposition \ref{prop:values}, $p_1:=G_1(x_S)/\psi_r(x_S)$ and $p_2:=L_2(x_S)/\psi_r(x_S)$. Also $v_1\in C({\cI})$ with $v_1\in W^{2,\infty}_{loc}(\underline{x},x_2)$ and $v_2\in C({\cI})$ with $v_2\in W^{2,\infty}_{loc}(x_1,\overline{x})$.
\end{prop}


\subsection{An alternate sufficient condition}
\label{relax}
In this section we consider similar preferences for the two players, in the sense that both stopping cost functions are drawn from the class $\cA_1$ (or both drawn from $\cA_2$). This implies that we must necessarily drop part \emph{(ii)} of Assumption \ref{ass:flessi}. Throughout the section we assume again that $\overline{x}$ is natural and $\underline{x}$ is either natural 
or entrance-not-exit. We will refer below to the stopping problems
\begin{align}\label{freeOS}
\inf_{\rho}\EE_x\left[e^{-r\rho}G_i(X_\rho)\right],\qquad i=1,2.
\end{align}

\begin{prop}\label{prop:relax}
Let $(i)$ and $(iii)$ of Assumption \ref{ass:flessi} hold. Assume now that $G_1, G_2 \in\cA_1$, with $\inf_{x\in\cI}G_i(x)<0$. Then, recalling \eqref{st-times0}, there exist $x'_i \in \cI$, $i=1,2$ such that $\rho^{1,\infty}_i:=\tau_1(x'_i)$, $i=1,2$ are optimal for \eqref{freeOS}. 

Moreover in the game \eqref{functional0}-\eqref{functional1} we have:
\begin{enumerate}
\item if $x'_1>x'_2$ then the couple $(\rho^{1,\infty}_1,+\infty)$ realises a Nash equilibrium;
\item if $x'_1<x'_2$ then the couple $(+\infty,\rho^{1,\infty}_2)$ realises a Nash equilibrium;
\item if $x'_1=x'_2$, then both the couples $(\rho^{1,\infty}_1,+\infty)$ and $(+\infty,\rho^{1,\infty}_2)$ realise Nash equilibria.
\end{enumerate}  
\end{prop}
\begin{proof}
For $i=1,2$ the existence of $x'_i$ may be easily verified since \eqref{freeOS} is an optimal stopping problem of the type studied in  \cite{DayKar} (notice that $\inf_{x\in\cI}G_i(x)<0$, $i=1,2$ guarantees that the trivial choice $\rho=+\infty$, $\PP_x$-a.s.~is not optimal in \eqref{freeOS}). The geometric solution method described therein may therefore be applied, namely the construction of the largest {\em non-positive} convex function dominated by $\hat{G}_i$. Under the current assumptions, $\hat{G}_i$ has a unique negative minimum at $\overline{y}_i\in(0,\hat{y}_i)$ (recall Definition \ref{def:inflection}). It follows that the stopping sets for problems \eqref{freeOS} are of the form $(\underline{x},x'_i]$ with $x'_i:=F^{-1}_r(\overline{y}_i)$ for $i=1,2$. Note that a separate verification argument is not required in this case, since the sufficiency of this construction is proved in \cite{DayKar}.

Next we establish the equilibria. We only consider the case $x'_1>x'_2$ as the other ones are analogous. Let us start by assuming that $P_1$ stops at $\rho^{1,\infty}_1$ and analyse $P_2$'s best reply. When the game is started at $x\le x'_1$, $P_2$ can either stop and incur a cost $G_2(x)$, or continue (i.e.~picking any $\tau_2>0$) and incur a cost $L_2(x)$. Hence the payoff for $P_2$ is $L_2(x)\mathds{1}_{\{\tau_2>0\}}+G_2(x)\mathds{1}_{\{\tau_2=0\}}$, which is clearly minimised by choosing $\tau_2>0$ a.s., since $L_2<G_2$.
Hence $\cJ_2(\rho^{1,\infty}_1,\tau_2,x)=L_2(x)$ for all $x\le x'_1$ and any $\tau_2>0$, and in particular this is true for $\tau_2=+\infty$.

Now for $x>x'_1$, $P_2$ is faced with the optimisation problem
\begin{align}\label{u1}
u(x):=\inf_{\tau}\EE_x\left[e^{-r\tau}G_2(X_\tau)\mathds{1}_{\{\tau\le\rho^{1,\infty}_1\}}+e^{-r\rho^{1,\infty}_1}L_2(X_{\rho^{1,\infty}_1})\mathds{1}_{\{\tau>\rho^{1,\infty}_1\}}\right].
\end{align}
Noticing that $G_2/\phi_r$ has the same monotonicity properties as $\hat G_2$ (Appendix \ref{app:convex}), and that the game is terminated by $P_1$ if $X$ hits $x'_1$ it is clear that 
\begin{align*}
\EE_x\left[e^{-r\tau}G_2(X_\tau)\mathds{1}_{\{\tau\le\rho^{1,\infty}_1\}}\right]\ge 
&\inf_{z\ge x'_1}\left(\frac{G_2}{\phi_r}\right)(z)\EE_x\left[e^{-r\tau}\phi_r(X_\tau)\mathds{1}_{\{\tau\le \rho^{1,\infty}_1\}}\right]\\
=& \left(\frac{G_2}{\phi_r}\right)(x'_1)\EE_x\left[e^{-r\tau}\phi_r(X_\tau)\mathds{1}_{\{\tau\le \rho^{1,\infty}_1\}}\right].
\end{align*}
Since also $L_2<G_2$ and $(e^{-rt}\phi_r(X_t))_{t\ge 0}$ is a martingale, we obtain the following lower bound:
\begin{align}\label{lb}
&u(x)\nonumber\\
&\ge \inf_{\tau}\left(\left(\frac{G_2}{\phi_r}\right)(x'_1)\EE_x\left[e^{-r\tau}\phi_r(X_\tau)\mathds{1}_{\{\tau\le \rho^{1,\infty}_1\}}\right]+\left(\frac{L_2}{\phi_r}\right)(x'_1)\EE_x\left[e^{-r\rho^{1,\infty}_1}\phi_r(X_{\rho^{1,\infty}_1})\mathds{1}_{\{\tau>\rho^{1,\infty}_1\}}\right]\right)\nonumber\\
&\ge\inf_{\tau}\left(\left(\frac{L_2}{\phi_r}\right)(x'_1)\EE_x\left[e^{-r(\tau\wedge\rho^{1,\infty}_1)}\phi_r(X_{\tau\wedge\rho^{1,\infty}_1})\right]\right)=\left(\frac{L_2}{\phi_r}\right)(x'_1)\phi_r(x).
\end{align}
Taking $\tau=+\infty$, $\PP$-a.s.~in \eqref{u1} (and recalling \eqref{limit}),~it now follows from the Laplace transforms \eqref{eqn:lapl} that the lower bound above is attained. Hence it is optimal for $P_2$ to choose $\tau=+\infty$ $\PP_x$-a.s.~for all $x>x'_1$.

In conclusion we have shown that $\tau=+\infty$ is a best reply of $P_2$ to $P_1$'s stopping rule $\rho^{1,\infty}_1$. Since $P_1$'s best reply to $\tau=+\infty$ is by definition $\rho^{1,\infty}_1$, we have reached an equilibrium with $(\rho_1^{1,\infty},+\infty)$ as claimed. 
\end{proof}

\begin{remark}
It is not difficult to check that under $(i)$ and $(iii)$ of Assumption \ref{ass:flessi}, letting $G_1, G_2 \in\cA_2$, there exist $x'_i \in \cI$, $i=1,2$ such that $\rho^{2,\infty}_i:=\tau_2(x'_i)$, $i=1,2$ are optimal for \eqref{freeOS}. Moreover in the game \eqref{functional0}-\eqref{functional1} we have:
\begin{enumerate}
\item if $x'_1>x'_2$ then the couple $(+\infty,\rho^{2,\infty}_2)$ realises a Nash equilibrium;
\item if $x'_1<x'_2$ then the couple $(\rho^{2,\infty}_1,+\infty)$ realises a Nash equilibrium; 
\item if $x'_1=x'_2$, then both the couples $(+\infty,\rho^{2,\infty}_2)$ and $(\rho^{2,\infty}_1,+\infty)$ realise Nash equilibria.
\end{enumerate}  
\end{remark}


\bigskip

\textbf{Acknowledgments.} The authors thank two anonymous referees for their pertinent and useful comments, which helped to substantially improve the paper.


\appendix

\section{Appendix}
\label{app:os1d}
\renewcommand{\theequation}{A-\arabic{equation}}

\subsection{Convexity of $\hat{H}$}
\label{app:convex}

We show here that $\hat{H}$ of \eqref{def:hat} is strictly convex at $y>0$ if and only if $(\LL_X-r)H(x)>0$ at $x=F^{-1}_r(y)$. We simply work out explicitly calculations indicated by \cite[Sec.~6]{DayKar}. For $y=F_r(x)$ it is obvious that
$\hat{H}'(y)=g(x)$ with $g(x):=(H/\phi_r)'(x)/F'_r(x)$ so that $\hat{H}''(y)=g'(x)/F'_r(x)$. Since $F_r$ is strictly increasing, we only need to evaluate $g'(x)$. This can be easily done by observing that
\begin{align*}
F'_r(x)=\frac{(\psi_r'\phi_r-\psi_r\phi'_r)(x)}{(\phi_r)^2(x)}=W\frac{S'(x)}{(\phi_r)^2(x)}\quad\text{and}\quad g(x)=\frac{(H'\phi_r-H\phi'_r)(x)}{W\,S'(x)}
\end{align*}
from which we get
\begin{align*}
g'(x)=\frac{\phi_r(x)(S'H''-S''H')(x)}{W\,(S')^2(x)}-\frac{H(x)(S'\phi_r''-S''\phi_r')(x)}{W\,(S')^2(x)}.
\end{align*}
Now we use that $S''(x)=-2\mu(x)S'(x)/\sigma^2(x)$ to obtain
\begin{align*}
g'(x)=\frac{2}{W\,\sigma^2(x)(S')(x)}\Big[\phi_r(x)\LL_X H(x)-H(x)\LL_X \phi_r(x) \Big]=\frac{2\phi_r(x)}{W\,\sigma^2(x)(S')(x)}(\LL_XH-rH)(x),
\end{align*}
where in the last equality we have used that $\LL_X\phi_r=r\phi_r$. The last expression proves the claim and we remark that the result holds even if $r=r(x)$ is state dependent.

\subsection{Proofs of some lemmas}
\label{app:lem}
\begin{proof}[Proof of Lemma \ref{lemm:concavity}.]
Recalling the notation of Section \ref{sec:H}, let $H\in \cA_1$. From Lemma \ref{lem:conv} we immediately get $i)$. We notice that indeed $F_r(\underline{x}+)=0$ and $F_r(\overline{x}-)=+\infty$ due to \eqref{psiphiproperties1} and the limit at zero of $\hat{H}$ is verified from the definition of $\cA$. 

If we now show that
\begin{equation}
\label{ab}
\text{(a)}\,\, \lim_{y \uparrow \infty}\hat{H}(y) = +\infty, \qquad \text{and} \qquad \text{(b)}\,\, \lim_{y \downarrow 0}\hat{H}'(y) =-\infty,
\end{equation}
we can then conclude parts $ii)$ and $iii)$.

First we prove (a) above. By the definition of $\cA_1$, for fixed $\delta > 0$ there exists $\eps_\delta>0$ such that $h(z) \leq -\varepsilon_{\delta}$ for any $z \in [x_h + \delta, \overline x)$. Moreover, for any $x\in[{x}_h + \delta, \overline{x})$, \eqref{resolvent2} implies
\begin{eqnarray}
\label{limit1}
H(x) & \hspace{-0.25cm} = \hspace{-0.25cm} & -W^{-1}\bigg[\phi_r(x)\int_{\underline{x}}^{{x}_h + \delta} \psi_r(z)h(z) m'(z)dz + \phi_r(x)\int_{x_h + \delta}^{x} \psi_r(z)h(z) m'(z)dz  \nonumber \\
&& \hspace{1cm} + \psi_r(x)\int_x^{\overline{x}} \phi_r(z)h(z) m'(z)dz\bigg] \\
& \hspace{-0.25cm} \geq \hspace{-0.25cm} &-W^{-1}\bigg[ \phi_r(x) C_{\delta} - \varepsilon_{\delta} \phi_r(x) \int_{x_h + \delta}^{x} \psi_r(z) m'(z)dz - \varepsilon_{\delta} \psi_r(x)\int_x^{\overline{x}} \phi_r(z) m'(z)dz\bigg] \nonumber 
\end{eqnarray}
with $C_{\delta} := \int_{\underline{x}}^{x_h + \delta} \psi_r(z)h(z) m'(z)dz$.
Using \eqref{psiphiproperties3} we have 
$$\int_{x_h + \delta}^{x} \psi_r(z) m'(z)dz = \frac{1}{r}\left[\frac{\psi'_r(x)}{S'(x)} - \frac{\psi'_r(x_h + \delta)}{S'(x_h + \delta)}\right]$$
and by using \eqref{psiphiproperties2} also, we obtain
$$\int_x^{\overline{x}} \phi_r(z) m'(z)dz = - \frac{1}{r}\frac{\phi'_r(x)}{S'(x)}.$$
Substituting these into \eqref{limit1},
the right hand side of \eqref{limit1} is equal to
\begin{equation}
\label{limit3}
-W^{-1} \left[C_{\delta} + \frac{\varepsilon_{\delta}}{r}\frac{\psi'_r(x_h + \delta)}{S'(x_h + \delta)}\right]\phi_r(x) + \frac{\varepsilon_{\delta}}{r},
\end{equation}
and so we have
\begin{equation}
\label{limit4}
\frac{H(x)}{\phi_r(x)} \geq -W^{-1} \left[C_{\delta} + \frac{\varepsilon_{\delta}}{r}\frac{\psi'_r(x_h + \delta)}{S'(x_h + \delta)}\right] + \frac{\varepsilon_{\delta}}{r\phi_r(x)}.
\end{equation}
Using \eqref{psiphiproperties1} we obtain
$$\lim_{x \uparrow \overline{x}}\frac{H(x)}{\phi_r(x)} = + \infty,$$
and since $\lim_{y\uparrow\infty}F^{-1}_r(y)=\overline{x}$, we have established part (a).

To prove (b) let $\delta>0$, take $x < x_h-\delta$, and let $y=F_r(x)$. 
Note that since $H\in\cA_1$ there exists $\eps_\delta>0$ such that $h(z)\ge \eps_\delta$ for $z\in(\underline x,x_h-\delta]$ and we obtain
\begin{align}\label{eq:zero1}
\hat{H}'(y)=&-\frac{1}{W}\Big[\int^{\overline{x}}_{x_h-\delta}\phi_r(z)h(z)m'(z)dz+
\int^{x_h-\delta}_x\phi_r(z)h(z)m'(z)dz\Big]\\
\le&-\frac{1}{W}\Big[\int^{\overline{x}}_{x_h-\delta}\phi_r(z)h(z)m'(z)dz+\eps_\delta\int_x^{x_h-\delta}\phi_r(z)m'(z)dz\Big]\nonumber \\
=& -W^{-1}\left[\int^{\overline{x}}_{x_h-\delta}\phi_r(z)h(z)m'(z)dz+\frac{\eps_\delta}{r}\Big(\frac{\phi'_r(x_h-\delta)}{S'(x_h-\delta)}-\frac{\phi'_r(x)}{S'(x)}\Big)\right],\nonumber
\end{align}
where the first line follows from \eqref{res3} and the chain rule and the third line by \eqref{psiphiproperties3}. Then, letting $y\downarrow0$ (equivalently $x\downarrow\underline{x}$) and using \eqref{psiphiproperties2}, we conclude $\hat{H}'(0+)=-\infty$. The case $H \in \cA_2$ follows by symmetric arguments.
\end{proof}

\begin{proof}[Proof of Lemma \ref{lem:aux1}]
Problem \eqref{eq:auxOS} is the same as the one in Appendix \ref{app:firstOS} below with $x_o=\underline{x}$ therein. Once we prove existence and uniqueness of $y^\infty_2$ then optimality of $\tau^\infty_2$ follows from Proposition \ref{prop:app-OS}.

The equation $\hat{G}'_2(y)y-\hat{G}_2(y)=0$ with $y>\hat{y}_2$ expresses the tangency condition for a straight line passing through the origin and tangent to $\hat{G}_2$ at a point in $(\hat{y}_2,+\infty)$. If a solution to that equation exists then the convexity of $\hat{G}_2$ in $(\hat{y}_2,+\infty)$ implies that it must be unique. For the existence it is sufficient to observe that
\begin{align*}
\hat{G}'_2(\hat{y}_2)\hat{y}_2< \int_0^{\hat{y}_2}\hat{G'_2}(s)ds=\hat{G}_2(\hat{y}_2)
\end{align*}
since $\hat{G}_2$ is strictly concave in $(0,\hat{y}_2)$. Recalling \eqref{eq:lim} we get $\lim_{y\to\infty}[\hat{G}'_2(y)y-\hat{G}_2(y)]=+\infty$ and therefore there exists a unique $y_2^\infty\in(\hat{y}_2,+\infty)$. 
\end{proof}

\subsection{Some remarks on state dependent discounting}
\label{sec:rdisc}

Here we illustrate the case of a state dependent discount rate $(r(X_t))_{t \geq 0}$. In this setting the payoffs \eqref{functional0} become:
 \begin{align}
\label{functional2}
&\mathcal{J}_1(\tau_1,\tau_2;x):=\EE_x\Big[e^{-\int_0^{\tau_1}r(X_t)dt}G_1(X_{\tau_1})\mathds{1}_{\{\tau_1 < \tau_2\}} + e^{-\int_0^{\tau_2}r(X_t)dt}L_1(X_{\tau_2})\mathds{1}_{\{\tau_2 \leq \tau_1\}}\Big],\\ 
&\mathcal{J}_2(\tau_1,\tau_2;x):=\EE_x\Big[e^{-\int_0^{\tau_2}r(X_t)dt}G_2(X_{\tau_2})\mathds{1}_{\{\tau_2 \le \tau_1\}} + e^{-\int_0^{\tau_1}r(X_t)dt}L_2(X_{\tau_1})\mathds{1}_{\{\tau_1 \leq \tau_2\}}\Big].
\end{align}
In order to extend the methodology applied above, we make sufficient assumptions on $r$ to ensure the existence of strictly monotonic and strictly positive fundamental solutions $\phi_r$, $\psi_r$ to the ODE
\begin{align}
\tfrac{1}{2}\sigma^2(x)f''(x)+\mu(x)f'(x)-r(x)f(x)=0, \quad x\in\cI.
\end{align}
In particular we assume that $r(x)$ is bounded, continuous and strictly positive for $x\in\cI$. In this case we again have
\begin{align}
\EE_x\big[e^{-\int_0^{\tau(y)}r(X_t)dt}\big]=\left\{
\begin{array}{ll}
\frac{\psi_r(x)}{\psi_r(y)}, & x<y,\\[+5pt]
\frac{\phi_r(x)}{\phi_r(y)}, & x>y,
\end{array}
\right.
\end{align}
for $x,y\in\cI$ and $\tau(y):=\inf\{t\ge0:X_t=y\}$ (see \cite{Day}, Prop.~2.1).
The limits at the endpoints of the domain $\cI$ of functions $\phi_r$, $\psi_r$, $\phi'_r/S'$ and $\psi'_r/S'$ remain the same as in the previous sections, depending on whether $\underline{x}$ is natural, entrance-not-exit or exit-not-entrance. Instead of the expressions \eqref{psiphiproperties3} we must now consider their generalisation (see par.~9 and 10, Ch.~2 of \cite{BS})
\begin{equation}
\label{psiphiproperties3b}
\frac{\psi'_r(b)}{S'(b)} - \frac{\psi'_r(a)}{S'(a)}= \int_{a}^{b}r(y)\psi_r(y)m'(y) dy, \qquad \frac{\phi'_r(b)}{S'(b)}-\frac{\phi'_r(a)}{S'(a)} = \int_{a}^{b}r(y)\phi_r(y)m'(y) dy,
\end{equation}
for $\underline{x}<a<b<\overline{x}$.

It is then easy to see that all the arguments that we have used for the construction of Nash equilibria in the above sections can be repeated for state dependent discounting and all the results carry over to this setting with no additional difficulties. In particular one should notice that positivity and boundedness of $r(\cdot)$ allow us to find bounds similar to those that led to some of our key inequalities (e.g.~\eqref{limit3} and \eqref{eq:zero1}); for example, setting $\overline{r}:=\sup_{z\in\cI}r(z)$ the second term in the first equality of \eqref{eq:zero1} can be bounded from below as follows
\begin{align*}
\int_x^{x_h-\delta}\phi_r(z)m'(z)dz\ge \frac{1}{\overline{r}}\int_x^{x_h-\delta}r(z)\phi_r(z)m'(z)dz=\frac{1}{\overline{r}}
\Big(\frac{\phi'_r(x_h-\delta)}{S'(x_h-\delta)}-\frac{\phi'_r(x)}{S'(x)}\Big)
\end{align*}
and the rest of the proof follows in the same way also with state dependent discounting.

We also remark that the argument used to infer convexity and concavity of the transformed functions $\hat{H}$ in Lemma \ref{lemm:concavity} and \ref{lemm:concavity2} holds in the same form, i.e.~$\hat{H}(y)$ is strictly convex if and only if $\tfrac{1}{2}\sigma^2(x)H''(x)+\mu(x)H'(x)-r(x)H(x)>0$ with $y=F_r(x)$.

\subsection{Two Useful Optimal Stopping Problems}
\label{2usefulOS}

The proof of Theorem \ref{thm:construction} involves solving, for each player in turn, an optimal stopping problem whose stopping cost function depends on the strategy of the other player. Our approach to such problems is inspired by a characterisation via convex analysis due to Dynkin, later developed in \cite{DayKar}. Since it is beyond the scope of the present paper to develop a complete theory for such problems, however, we adopt the following hybrid approach. A geometric construction similar to that employed in the latter references is first used to propose a candidate stopping region and payoff. This candidate solution is then verified in a second step. In this way we obtain a convenient geometric characterisation of the stopping set and payoff for the particular optimal stopping problems encountered in Theorem \ref{thm:construction}.

\subsubsection{A First Optimal Stopping Problem}
\label{app:firstOS}

Recall Definition \ref{def:sets} along with the notation of \eqref{def-F} and \eqref{def:hat}, and consider a function $G \in \mathcal{A}_2$. Denote by $\hat{x} \in \cI$ the unique point at which $\LL_X G-rG$ changes its sign and take $x_o \in \cI$ with $x_o< \hat{x}$. Let us introduce the infinite time horizon optimal stopping problem with value function
\begin{equation}
\label{app:OS}
V_o(x) : = \inf_{\tau \in \mathcal{T}}\EE_x\Big[e^{-r\tau}G(X_{\tau})\mathds{1}_{\{\tau \le \tau_o\}} + L(X_{\tau_o}) e^{-r\tau_o}\mathds{1}_{\{\tau_o<\tau\}}\Big], 
\end{equation}
where $L(x) < G(x)$ for all $x\in\cI$ and $\tau_o := \inf\{t \geq 0 : X^x_t \leq x_o\}$. 

First we notice that if $x \le x_o$, picking any $\tau$~gives a payoff equal to $L(x)\mathds{1}_{\{\tau>0\}}+G(x)\mathds{1}_{\{\tau=0\}}$. The latter is minimised by choosing an arbitrary $\tau>0$ a.s., and
\begin{align}\label{Vo}
V_o(x)=L(x),\quad\text{for $x\le x_o$}.
\end{align}
Further, the next standard argument shows that it is never optimal to stop for $x<\hat{x}$ since $\LL_X G-rG<0$. In fact if $x<\hat{x}$, the sub-optimal stopping time $\hat{\tau}:=\inf\{t\ge 0: X^x\ge \hat{x}\}$ gives 
\begin{align*}
V_o(x)\le& \EE_x\Big[e^{-r\hat{\tau}}G(X_{\hat{\tau}})\mathds{1}_{\{\hat{\tau} \le \tau_o\}} + L(X_{\tau_o}) e^{-r\tau_o}\mathds{1}_{\{\tau_o<\hat{\tau}\}}\Big]\\
\le& \EE_x\Big[e^{-r(\hat{\tau}\wedge\tau_o)}G(X_{\hat{\tau}\wedge\tau_o})\Big]<G(x)
\end{align*}
where the last inequality is obtained by using Dynkin's formula and $\LL_X G-rG<0$ in $(\underline{x},\hat{x})$. 

It then follows that if an optimal stopping time $\tau_*$ exists then $\PP_x(\tau_*=\tau_o)=0$ for all $x\in\cI$, since $x_o<\hat x$. Hence in \eqref{app:OS} we could replace the event $\{\tau \le \tau_o\}$ by the event $\{\tau < \tau_o\}$ with no loss of generality, thus avoiding potential problems concerning continuity of the value function at $x_o$.

Set $y_o:= F_{r}(x_o)$ and define the function
\begin{equation}
\label{app:def-H}
Q(y):=\left\{
\begin{array}{ll}
\hat{L}(y), \quad 0 < y \leq y_o, \\[+8pt]
\hat{G}(y), \quad y > y_o.
\end{array}
\right.
\end{equation}
For $\hat{y}:=F_r(\hat{x})$ we argue as in Section \ref{app:convex} of this appendix and obtain that $\hat{G}$ is strictly concave in $(0, \hat{y})$ and strictly convex in $(\hat{y}, \infty)$. 

Let us consider the straight line $r_o(\cdot)$ which passes through the point $(y_o, Q(y_o))$ and is tangent to $Q$ at a point $y_*>\hat{y}:=F_r(\hat{x})$. Existence of $r_o$ can be easily proven due to convexity/concavity of $Q$ (Section \ref{app:convex}) and we leave it to the reader. This line is expressed as
\begin{equation}
\label{app:line}
r_o(y) = my + q, \quad y > 0,
\end{equation}
with
\begin{equation}
\label{app:mq}
\left\{
\begin{array}{ll}
m:=\frac{Q(y_*)-Q(y_o)}{y_* - y_o}, \\[+3pt]
q:= Q(y_o) - m y_o.
\end{array}
\right.
\end{equation}
By the convexity of $\hat{G}$ (and therefore of $Q$) in $(\hat{y},+\infty)$ the point $y_*$ is determined as the unique $y>\hat{y}$ that solves the tangency equation
\begin{align}\label{app:tan}
\frac{Q(y)-Q(y_o)}{y - y_o} = Q'(y).
\end{align}

\begin{prop}
\label{prop:app-OS} Let $G \in \cA_2$.
Assume there exists $y_* > \hat{y}$ solving \eqref{app:tan} (which is then unique). Recall \eqref{app:line} and \eqref{app:mq} and define $x_*:=F_r^{-1}(y_*)$ and the functions
\begin{equation}
\label{app:def-W}
W(y):=
\left\{
\begin{array}{ll}
\hat{L}(y), & 0 < y \leq y_o  \\[+3pt]
my + q, & y_o < y < y_* \\[+3pt]
\hat{G}(y), & y \ge y_*,
\end{array}
\right.
\end{equation}
and
\begin{equation}
\label{app:def-tildeV}
\widetilde{V}_o(x):= \phi_r(x) W(F_r(x)) =
\left\{
\begin{array}{ll}
L(x), & \underline{x} < x  \leq x_o  \\[+3pt]
m\psi_r(x) + q \phi_r(x), & x_o < x < x_* \\[+3pt]
G(x), & x_* \le x < \overline{x}.
\end{array}
\right.
\end{equation}
Then one has $\widetilde{V}_o \equiv V_o$ and $\tau_*:=\inf\{ t \geq 0: X_t \geq x_*\}$ is optimal for problem \eqref{app:OS}.
\end{prop}
\begin{proof}
If $x \leq x_o$ there is clearly nothing to prove thanks to \eqref{Vo}. Therefore, take $x>x_o$ and notice by \eqref{app:def-tildeV} that
$(\LL_X - r)\widetilde{V}_o(x) =0$ if $x \in (x_o, x_*)$. Moreover, by Section \ref{app:convex} we also have that $(\LL_X - r)\widetilde{V}_o(x) \geq 0$ if $x \in (x_*, \overline{x})$, since $y_* > \hat{y}$ and $\hat{G}$ is convex in $(\hat{y}, \infty)$. Also, by construction, $\widetilde{V}_o(x_*) = G(x_*)$, $\widetilde{V}^{'}_o(x_*) = G^{'}(x_*)$, $\widetilde{V}_o(x_o) = L(x_o)$ and $\widetilde{V}_o \leq G$, for any $x>x_o$. Since $\widetilde{V}_o \in W^{2,\infty}_{\text{loc}}((x_0, \overline{x}))$ we can apply It\^o-Tanaka's formula
to the process $(e^{-rt}\widetilde{V}_o(X^x_t))_{t\geq 0}$ on the time interval $[0,\tau \wedge \tau_o]$, for arbitrary $\tau \in \mathcal{T}$, and obtain
\begin{align}\label{ineq1}
\widetilde{V}_o(x)\le \EE_x \Big[e^{-r\tau\wedge\tau_o}\widetilde{V}_o(X_{\tau\wedge\tau_o})\Big]\le \EE_x\Big[e^{-r\tau}G(X_{\tau})\mathds{1}_{\{\tau \le \tau_o\}} + L(X_{\tau_o}) e^{-r\tau_o}\mathds{1}_{\{\tau > \tau_o\}}\Big]
\end{align}
and hence $\widetilde{V}_o\le V_o$. Then repeating the argument with $\tau=\tau_*$ we find 
\begin{align*}
\widetilde{V}_o(x)=\EE_x\Big[e^{-r\tau_*}G(X_{\tau_*})\mathds{1}_{\{\tau_* \le \tau_o\}} + L(X_{\tau_o}) e^{-r\tau_o}\mathds{1}_{\{\tau_* > \tau_o\}}\Big]
\end{align*}
and therefore $\widetilde{V}_o= V_o$ and $\tau_*$ is optimal.
\end{proof}
Notice that, when restricted to $[y_o, + \infty)$, the function $W$ is the largest convex function dominated by $Q$. The latter condition makes the result slightly different to the geometric characterisation in \cite{DayKar} (they have $\hat{G}=\hat L$ and then $W$ is non-positive, see also \cite{MoPa14}).


\subsubsection{A Second Optimal Stopping Problem}
\label{app:secondOS}

For the next optimal stopping problem we take the same setup as in Section \ref{app:firstOS}, with the modifications that $G \in \mathcal{A}_1$, $x_o > \hat{x}$ and $\tau_o := \inf\{t \geq 0 : X^x_t \geq x_o\}$. Again we recall that if an optimal stopping time $\tau_*$ exists then $\PP_x(\tau_*=\tau_o)=0$ for all $x\in\cI$ so that the indicator functions in \eqref{app:OS} may effectively be taken with strict inequalities only. As before we denote
\begin{equation}
\label{app:def-H-bis}
Q(y):=\left\{
\begin{array}{ll}
\hat{G}(y), & 0 < y < y_o , \\[+3pt]
\hat{L}(y), & y \geq y_o.
\end{array}
\right.
\end{equation} 
In contrast to the situation in Section \ref{app:firstOS}, in the present setting we will consider two possible geometries for this optimal stopping problem, in Propositions \ref{prop:app-OS-bis} and \ref{prop:app-OS-tris} respectively. This is necessary because it may in principle be optimal to stop anywhere in the interval $(\underline x,\hat x)$ and the geometry of $\hat G$ on $(0,\hat y)$ depends on the boundary behaviour of $X$ at $\underline x$, which we vary through the paper.

In what follows we write
\begin{align}
\bar m:= Q(y_o)/y_o
\end{align}
and recall that for a not non-singular lower boundary $\underline x$ we have $\psi_r(x)/\phi_r(x)\to0$ as $x\to\underline x$.

\begin{prop}\label{prop:app-OS-bis} 
Let $G\in\cA_1$. If 
\begin{align}
\label{case1}
Q(y)> \bar m y\qquad\text{for all $y\in(0,y_o)$}
\end{align}
then the function
\begin{equation}
\label{Vtildeo}
\widetilde{V}_o(x):=
\left\{
\begin{array}{ll}
\bar m \,\psi_r(x), & \underline{x} < x < x_o, \\[+3pt]
L(x), & x_o \leq x < \overline{x},
\end{array}
\right.
\end{equation}
is such that $\widetilde{V}_o \equiv V_o$ and, further, the stopping time $\tau_*:=+\infty$ is optimal for problem \eqref{app:OS}. 
\end{prop}
\begin{proof}
Clearly $\widetilde{V}_o\in C^2_b(\underline{x},x_o)$ and $\LL_x\widetilde{V}_o-r\widetilde{V}_o=0$ on $(\underline{x},x_o)$. By the same geometric arguments as in the proof of Proposition \ref{prop:app-OS}, and using \ref{case1}, we also have $\widetilde{V}_o(x)\le G(x)$. Hence by applying It\^o's formula for $x\in(\underline x,x_o)$ and any $\tau$ we get \eqref{ineq1} and therefore $\widetilde V_o(x)\le V_o(x)$. 

For fixed $\eps>0$, picking $\tau=\tau_\eps:=\inf\{t\ge0 \,:\, X_t\le \underline x+\eps\}$ we also find
\begin{align*}
\widetilde{V}_o(x)=&\EE_x\left[e^{-r(\tau_\eps\wedge\tau_o)}\widetilde{V}_o(X_{\tau_\eps\wedge\tau_o})\right]\\
=&\EE_x\left[e^{-r\tau_\eps}\widetilde{V}_o(X_{\tau_\eps})\mathds{1}_{\{\tau_\eps\le \tau_o\}}+e^{-r\tau_o}L(X_{\tau_o})\mathds{1}_{\{\tau_o< \tau_\eps\}}\right]\\
=&\bar m \frac{\psi_r(\underline x+\eps)}{\phi_r(\underline x+\eps)}\lambda_o(x)+\EE_x\left[e^{-r\tau_o}L(X_{\tau_o})\mathds{1}_{\{\tau_o< \tau_\eps\}}\right],
\end{align*}
with 
\begin{align*}
\lambda_o(x):=\frac{\psi_r(x)\phi_r(x_o)-\psi_r(x_o)\phi_r(x)}{(\psi_r/\phi_r)(\underline x +\eps)\phi_r(x_o)-\psi_r(x_o)}
\end{align*}
by \eqref{Vtildeo} and equation (4.3) in \cite{DayKar}.
Letting $\eps\to0$ we have $\tau_\eps\to\sigma_{\cI}$ and $\psi_r(\underline x+\eps)/\phi_r(\underline x+\eps)\to0$. Thus, taking limits and using dominated convergence and \eqref{limit}, we obtain 
\begin{align*}
\widetilde{V}_o(x)=\EE_x\left[e^{-r\tau_o}L(X_{\tau_o})\mathds{1}_{\{\tau_o< \infty\}}\right]\ge V_o(x),
\end{align*}
hence completing the proof.
\end{proof}

Because of the convexity of $\hat G$ on $(0,\hat y)$, if 
\begin{align}\label{case2}
\hat G(0+)>0\quad\text{and}\quad \inf_{y\in(0,y_o)} [\hat G(y)-\bar m y]\le0
\end{align}
then there exist two points $0<y_{*,1}\le y_{*,2}<\hat y$ such that the straight line $r_1(y):= \bar m_1 y$, with $\bar m_1:=Q(y_{*,1})/y_{*,1}$, is tangent to $Q(y_{*,1})$ while $y_{*,2}$ solves \eqref{app:tan}. In this case we have the following proposition, whose proof we omit due to its similarity to Proposition \ref{prop:app-OS}.

\begin{prop}
\label{prop:app-OS-tris} 
Let $G \in \cA_1$ be such that \eqref{case2} is satisfied and let $y_{*,1}$ and $y_{*,2}$ be the points described above. 
Then writing $x_{*,i}:=F_r^{-1}(y_{*,i})$ for $i=1,2$ and defining $m,q$ as in \eqref{app:mq}, the function
\begin{equation}
\label{app:def-tildeV-bis}
\widetilde{V}_o(x):=
\left\{
\begin{array}{ll}
\bar m_1 \psi(x), & \underline{x} < x < x_{*,1},\\[+3pt]
G(x), & x_{*,1} \le x \le x_{*,2},\\[+3pt]
m\psi_r(x) + q \phi_r(x), & x_{*,2} < x < x_o, \\[+3pt]
L(x), & x_o \leq x < \overline{x},
\end{array}
\right.
\end{equation}
is such that $\widetilde{V}_o \equiv V_o$. Further the stopping time $\tau_*:=\inf\{ s \geq 0: X_s \in [x_{*,1},x_{*,2}]\}$ is optimal for problem \eqref{app:OS}. 
\end{prop}
It is immediate to check that, when restricted to the domain $(0,y_o]$, the function $W(y):=(V_o/\phi_r)\circ F^{-1}_r(y)$ is the largest convex function dominated by $Q$. If $\hat G(0+)\le 0$ then $y_{*,1}$ clearly does not exist in $(0,y_o)$ and instead we state the following corollary, whose proof is left to the reader.
\begin{coroll}
Let $G \in \cA_1$ with $\hat G(0+)\le 0$ and assume $y_{*,2}$ exists as described above. Then the results of Proposition \ref{prop:app-OS-tris} hold with $x_{*,1}=\underline x$.
\end{coroll}



\begin{thebibliography}{99}

\bibitem{Alarioetal}\textsc{Alariot, M., Lepeltier, J.P., Marchal, B.}\ (1982). \textsl{Jeux de Dynkin}, in \textsl{Proceedings of the 2nd Bad Honnef Workshop on Stochastic Processes}, Lecture Notes in Control and Inform.\ Sci.\ pp.\ 23--32, Springer-Verlag, Berlin.

\bibitem{Alvarez2008} \textsc{Alvarez, L.}\ (2008). \textsl{A Class of Solvable Stopping Games}, Appl.\ Math.\ Optim.\ $58$, pp.\ 291--314.

\bibitem{Anderson}\textsc{Anderson, S. T., Friedman, D., Oprea, R.}\ (2010). \textsl{Preemption Games: Theory and Experiment}, Am.~Econ.~Rev.~100(4), pp. \ 1778--1803.

\bibitem{Attard1}\textsc{Attard, N.}\ (2015). \textsl{Non-Zero Sum Games of Optimal Stopping for Markov Processes}, Probability and Statistics Research Reports No.\ 1, School of Mathematics, The University of Manchester.

\bibitem{Attard2}\textsc{Attard, N.}\ (2017). \textsl{Nash Equilibrium in Non-Zero Sum Games of Optimal Stopping for Brownian Motion}, Adv.\ Appl.\ Probab.~49, pp.~430--445.

\bibitem{Bismut}\textsc{Bismut, J.-M.}~(1977). \textsl{Sur un Probl\`eme de Dynkin}, Z.~Warsch.~V.~Geb.~39, pp.~31--53.

\bibitem{BensoussanFriedman} \textsc{Bensoussan, A., Friedman, A.}~(1977) \textsl{Non-zero Sum Stochastic Differential Games with Stopping Times and Free-Boundary Problems}, Trans.~Amer.~Math.~Soc.~$231$, pp.\ 275--327.

\bibitem{BS}\textsc{Borodin, A.N., Salminen, P.}\ (2002). \textsl{Handbook of Brownian Motion-Facts and Formulae} 2nd edition. Birkh\"auser.

\bibitem{CvitanicKar}\textsc{Cvitanic, J., Karatzas, I.}\ (1996). \textsl{Backward SDEs with Reflection and Dynkin Games}, Ann.\ Probab.\ 24, pp.\ 2024--2056.

\bibitem{CattiauxLep}\textsc{Cattiaux, P., Lepeltier, J.P.}\ (1990). \textsl{Existence of a Quasi-Markov Nash Equilibrium for Nonzero-Sum Markov Stopping Games}, Stoch.\ Stoch.\ Rep.\ 30, pp.\ 85--103.

\bibitem{Day}\textsc{Dayanik, S.}~(2008). \textsl{Optimal Stopping of Linear Diffusions with Random Discounting}, Math.~Oper.~Res.~ $33(3)$. pp.~645--661.

\bibitem{DayKar}\textsc{Dayanik, S., Karatzas, I.}\ (2003). \textsl{On the Optimal Stopping Problem for One-Dimensional Diffusions}, Stochastic Process.~Appl.~$107(2)$, pp.\ 173--212.

\bibitem{DeAFe14}\textsc{De Angelis, T., Ferrari, G.}~(2014), \textsl{A Stochastic Partially Reversible Investment Problem on a Finite Time-horizon: Free-boundary Analysis}. Stochastic Process.~Appl.~$124(12)$, pp.~4080--4119.

\bibitem{Dieudonne}\textsc{Dieudonn\'e, J.}\ (1969). \textsl{Foundations of Modern Analysis}, Volume 1. Elsevier.

\bibitem{Dynkin01} \textsc{Dynkin, E.B.}(1969). \textsl{Game Variant of a Problem on Optimal Stopping}, Soviet.\ Math.\ Dokl.\ $10$, pp.\ 270--274.

\bibitem{Dynkin}\textsc{Dynkin, E.B., Yushkevich, A.A.}\ (1969). \textsl{Markov Processes: Theorems and Problems}. Plenum Press, New York.

\bibitem{EkstromVilleneuve}\textsc{Ekstr\"om, E., Villeneuve, S.}\ (2006). \textsl{On the Value of Optimal Stopping Games}, Ann.\ Appl.\ Probab.\ 16, pp.\ 1576--1596.

\bibitem{Etorneau}\textsc{Etorneau, E.}\ (1986). \textsl{R\'esolution d'un Probl\`eme de Jeu de Somme Non Nulle sur les Temps d'Arr\^et}, Th\`ese de 3-i\`eme cycle, Univ.\ Paris 6.

\bibitem{FT85}\textsc{Fudenberg, D., Tirole, J.}~(1985). \textsl{Preemption and Rent Equalization in the Adoption of New Technology}, Rev.~Financ.~Stud.~$52(3)$, pp.~383--401.

\bibitem{HamZha10} \textsc{Hamadene, S., Zhang, J.}~(2010). \textsl{The Continuous Time Nonzero-Sum Dynkin Game Problem and Application in Game Options}, SIAM J.~Control Optim.~$48$(5), pp.\ 3659--3669.

\bibitem{HamadeneHassani} \textsc{Hamadene, S., Hassani, M.}~(2014). \textsl{The Multi-player Nonzero-sum Dynkin Game in Continuous Time}, SIAM J.~Control Optim.~52(2), pp.~821--835.

\bibitem{ItoMcKean} \textsc{It\^o, K., McKean, Jr., H.P.}\ (1974). \textsl{Diffusion Processes and Their Sample Paths}. Springer Verlag, Berlin, Heidelberg and New York.

\bibitem{JYC}\textsc{Jeanblanc, M., Yor, M., Chesney, M.}\ (2009). \textsl{Mathematical Methods for Financial Markets}, Springer.

\bibitem{KS} \textsc{Karatzas, I., Shreve, S.E.}~(1998). \textsl{Brownian Motion and Stochastic Calculus} 2nd Edition. Springer.

\bibitem{KL09} \textsc{Karatzas, I., Li, Q.}~(2009). \textsl{BSDE Approach to Non-Zero-Sum Stochastic Differential Games of Control and Stopping}, Stochastic Processes, Finance and Control, pp.~105-153.

\bibitem{Kifer} \textsc{Kifer, Y.}~(2000). \textsl{Game Options}, Finance Stoch.\ 4, pp.~443-463.

\bibitem{Kypri}\textsc{Kyprianou, A.E.}~(2004). \textsl{Some Calculations for Israeli Options}, Finance Stoch.~8, pp.~73--86.

\bibitem{LarakiSolan}\textsc{Laraki, R., Solan, E.}\ (2005), \textsl{The Value of Zero-Sum Stopping Games in Continuous Time}, SIAM J.~Control Optim.~$43$, pp.\ 1913--1922.

\bibitem{LarakiSolan2}\textsc{Laraki, R., Solan, E.}\ (2013), \textsl{Equilibrium in 2-player Non Zero Sum Dynkin Games in Continuous Time}, Stochastics~85 (6), pp.\ 997--1014.

\bibitem{Mamer}\textsc{Mamer, J.W.}\ (1987), \textsl{Monotone Stopping Games}, J.\ Appl.\ Probab.~$24$, pp.\ 386--401.

\bibitem{MoPa14}\textsc{Moriarty, J., Palczewski, J.}~(2017), \textsl{Real Option Valuation for Reserve Capacity}, 	Eur. J. Oper. Res.~$257$, pp.\ 251--260.

\bibitem{Morimoto}\textsc{Morimoto, Y.}\ (1986), \textsl{Nonzero-sum Discrete Parameter Stochastic Games with Stopping Times}, Probab.\ Theory Related Fields~$72$, pp.\ 155--160.

\bibitem{Murto}\textsc{Murto, P.}~(2004), \textsl{Exit in Duopoly under Uncertainty}. RAND J.~Econ.~$35$, pp.~111--127.

\bibitem{Nagai} \textsc{Nagai, H.}~(1987). \textsl{Non Zero-Sum Stopping Games of Symmetric Markov Processes}, Probab.~Th.~Rel.~Fields $75$, pp.~487--497.

\bibitem{Ohtsubo}\textsc{Ohtsubo, Y.}\ (1987), \textsl{A Nonzero-Sum Extension of Dynkin Stopping Problem}, Math.\ Oper.\ Res.~$12$, pp.\ 591--607.

\bibitem{Ohtsubo2}\textsc{Ohtsubo, Y.}\ (1991), \textsl{On a Discrete-Time Nonzero-Sum Dynkin Problem with Monotonicity}, J.\ Appl.\ Probab.~$28$, pp.\ 466--472.

\bibitem{PeskirGame}\textsc{Peskir, G.}\ (2008). \textsl{Optimal Stopping Games and Nash Equilibrium}, Theory Probab.~Appl.\ 53(3), pp.\ 558--571.

\bibitem{Pham}\textsc{Pham, H.}\ (2009). \textsl{Continuous-time Stochastic Control and Optimization with Financial Applications}. Springer.

\bibitem{RiSte15} \textsc{Riedel, F., Steg, J.H.}~(2017). \textsl{Subgame-perfect Equilibria in Stochastic Timing Games}, J.~Math.~Econ.~72, pp.~36--50.

\bibitem{ShmayaSolan}\textsc{Shmaya, E., Solan, E.}\ (2004), \textsl{Two-player Nonzero-Sum Stopping Games in Discrete Time}, Ann.\ Probab.~$32$, pp.\ 2733--2764.

\bibitem{SirbuShreve}\textsc{Sirbu, M., Shreve, S.E}\ (2006). \textsl{A Two-Person Game for Pricing Convertible Bonds}, SIAM J.~Control Optim.\ 45(4), pp.\ 1508--1539.

\bibitem{Steg2016}\textsc{Steg, J.-H.}\ (2016). \textsl{Preemptive Investment under Uncertainty}, arXiv: 1511.03863.

\bibitem{TouziVieille}\textsc{Touzi, N., Vieille, N.}~(2002). \textsl{Continuous-time Dynkin Games with Mixed Strategies}, SIAM J.~Control Optim.\ 41(4), pp.\ 1073--1088.

\end{thebibliography}
\end{document}